\documentclass[reqno]{amsart}

\usepackage{mathabx, amscd}
\usepackage[utf8]{inputenc}
\usepackage{amsmath}
\usepackage{amsfonts}
\usepackage{amsthm}
\usepackage{amssymb}
\usepackage[normalem]{ulem}
\usepackage{enumerate}
\usepackage{tikz}
\usepackage{xcolor}
\usepackage[hidelinks]{hyperref}
\usepackage{mathrsfs}
\hypersetup{colorlinks,breaklinks,
             urlcolor=blue,
             linkcolor=blue}

\theoremstyle{plain}
\newtheorem{thm}{Theorem}[section]
\newtheorem{lem}[thm]{Lemma}
\newtheorem{pro}[thm]{Proposition}
\newtheorem{cor}[thm]{Corollary}
\theoremstyle{definition}
\newtheorem{defn}[thm]{Definition}
\newtheorem{rem}[thm]{Remark}

\newtheorem{exa}[thm]{Example}

\newcommand{\citep}{\cite}
\newcommand{\F}{\mathbb{F}}
\newcommand{\Fqm}{\F_{q^m}}
\newcommand{\Fq}{\F_q}
\newcommand{\Id}{\mathrm{Id}}

\newcommand{\<}{\left<}
\renewcommand{\>}{\right>}

\newcommand{\NN}{\mathbb{N}}
\newcommand{\Np}{\mathbb{N}^+}

\renewcommand{\r}{\rho}
\newcommand{\x}{\mathbf{x}}

\newcommand{\R}{\mathcal{R}}

\newcommand{\B}{\mathcal{B}}
\newcommand{\AC}{\mathcal{A}}
\newcommand{\M}{\mathbf{M}}
\newcommand{\T}{\tau}
\renewcommand{\aa}{\mathbf{a}}

\newcommand{\I}{\mathcal{I}}
\newcommand{\0}{\mathbf{0}}

\newcommand{\ZZ}{\mathbb{Z}}

\newcommand{\SE}{\Sigma(E)}
\newcommand{\GrE}{\mathbb{G}_r(E)}
\newcommand{\GkE}{\mathbb{G}_k(E)}

\title{Shellability and Homology of $q$-complexes and $q$-matroids}
	\author{Sudhir R. Ghorpade}
	\address{Department of Mathematics, 
		Indian Institute of Technology Bombay,\newline \indent
		Powai, Mumbai 400076, India}
\email{\href{mailto:srg@math.iitb.ac.in}{srg@math.iitb.ac.in}}
\thanks{Sudhir Ghorpade is partially supported by DST-RCN grant INT/NOR/RCN/ICT/P-03/2018 from the Dept. of Science \& Technology, Govt. of India, MATRICS grant MTR/2018/000369 from the Science \& Eng. Research Board, and IRCC award grant 12IRAWD009 from IIT Bombay.}
\author{Rakhi Pratihar}
		\address{Department of Mathematics and Statistics, 
		UiT - The Arctic University of Norway, 
\newline \indent
N-9037 Troms{\o}, Norway}
\email{\href{mailto:pratihar.rakhi@gmail.com}{pratihar.rakhi@gmail.com}}
\thanks{During the course of this work, Rakhi Pratihar was  supported by a doctoral fellowship at IIT Bombay from the 
University Grant Commission, Govt. of India (Sr. No. 2061641156). Currently, she is supported by Grant 280731 from the Research Council of Norway.}
\author{Tovohery H. Randrianarisoa}
\address{Department of Mathematical Sciences, 
Florida Atlantic University, \newline \indent
Boca Raton, FL 33431, USA}
\email{\href{mailto:tovo@aims.ac.za}{tovo@aims.ac.za}}
\thanks{During the course of this work Tovohery Randrianarisoa was supported by a postdoc fellowship at IIT Bombay from the Swiss National Science Foundation Grant No. 181446.}

\subjclass[2010]{05B35, 05A30, 52B22, 55N99}
\begin{document}

\maketitle

\begin{abstract}
		We consider a $q$-analogue of abstract simplicial complexes, called $q$-complexes, and discuss the  notion of shellability for such complexes. It is shown that $q$-complexes formed by independent subspaces of a $q$-matroid are shellable. 
		Further, we explicitly determine the homology of $q$-complexes corresponding to uniform $q$-matroids. We also outline some partial results concerning the determination of homology of arbitrary 
		shellable $q$-complexes.
\end{abstract}

\section{Introduction}
Shellability is an important and useful notion in combinatorial topology and algebraic combinatorics. 
Recall that an (abstract) simplicial complex $\Delta$ is said to be \emph{shellable} if it is pure (i.e., all its facets have the same dimension) and there is a linear ordering $F_1, \dots , F_t$ of its facets 
such that for each $j=2, \dots , t$, the complex 
$\langle F_j \rangle \cap \langle F_1, \dots F_{j-1} \rangle$ is generated by a nonempty set of maximal proper faces of $F_j$.  Here for $i=1, \dots , t$, by $\langle F_1, \dots , F_i \rangle$ we denote the complex generated by $F_1, \dots , F_i$, i.e., the smallest simplicial complex containing $F_1, \dots , F_i$.  

From a topological point of view, a shellable simplicial complex is like a wedge of spheres. In particular, the reduced homology groups are well understood. 
Shellable simplicial complexes are of importance in commutative algebra partly because their Stanley-Reisner rings (over any field) are  Cohen-Macaulay. 
Gr\"obner deformations of coordinate rings of several classes of algebraic varieties can be viewed as Stanley-Reisner rings of some simplicial complexes. 
Thus showing that these complexes are shellable becomes an effective way of establishing Cohen-Macaulayness of the corresponding coordinate rings. 
Important classes of simplicial complexes that are known to be shellable include boundary complex of a convex polytope, order complex of a ``nice'' poset (or more precisely, a bounded, locally upper semimodular poset), and matroid complexes, i.e., complexes formed by the independent subsets of matroids. For relevant background and proofs of these assertions, we refer to the monographs \cite{Stan, BH, GSSV} and the survey article of  Bj\" orner \cite{AB1}. 

We are interested in a $q$-analogue of some of these notions and results, wherein finite sets are replaced by finite-dimensional vector spaces over the finite field $\Fq$. One of our motivation comes from the recent work of Jurrius and Pellikaan \cite{JP} where the notion of a $q$-matroid is introduced and several of its properties are studied. (See also Crapo \cite{Cr} and Terwilliger \cite{Ter} where more general notions are studied.) The notion of a simplicial complex admits a straightforward $q$-analogue, and this goes back at least to Rota \cite{Rota}. 
Alder \cite{Alder} studied $q$-complexes in his thesis and defined when a $q$-complex is shellable. A natural question therefore is whether the $q$-complex of independent subspaces of a $q$-matroid is shellable. We will show in this paper that the answer is affirmative. 

Next, we consider the question of determining the homology of shellable $q$-complexes. This appears to be much harder than the classical case, and we are able to make partial progress here by way of explicitly determining the homology of $q$-spheres 
as well as the more general class of $q$-complexes formed by independent subspaces of  uniform $q$-matroids. We also describe the homology of a 
shellable $q$-complex provided it satisfies an additional hypothesis. 
A basic stumbling block (pointed out in \cite{JP} already)  is that  the notions of difference (of two sets) and complement (of a subset of a given set) do not have an obvious and unique analogue in the context of subspaces.

Our other motivation is from coding theory and the work of Johnsen and Verdure \cite{JV1} where to a $q$-ary linear code (or more generally, to a matroid), one can associate a fine set of invariants, called its Betti numbers. These are 
obtained by looking at a minimal graded free resolution of the Stanley-Reisner ring of a simplicial complex that  corresponds to the vector matroid associated to the parity check matrix of the given linear code. The question that arises naturally is whether something like Betti numbers can be defined in the context of rank metric codes, or more generally, for $q$-matroids as in \cite{JP} or going even further, for the $(q,m)$-polymatroids studied in \cite{S, BMS, GJ} or the $q$-polymatroids studied in \cite{GJLR}.  We were led to the study of shellability and homology of $q$-complexes, and especially, complexes associated to $q$-matroids with a view toward a possible topological approach to the above question. However, the question of arriving at a suitable notion of Betti numbers of rank metric codes is very far from being answered and  at the moment, the  musings above are more like a pie in the sky. 
 
This paper is organized as follows. In the next section, we collect some preliminaries and recall definitions of basic concepts such as $q$-complexes and $q$-matroids. 
In Section \ref{Sec4}, we outline a procedure called ``tower decomposition'' that 
provides a useful way to order subspaces in a $q$-complex. 
The notion of shellability for $q$-complexes is reviewed in Section \ref{Sec5} and the shellability of $q$-matroid complexes is also established in this section. 
Next, we explicitly determine the homology of $q$-spheres, and more generally,  the homology of the so called uniform $q$-complexes in Section \ref{Sec6}. Finally, our results on the homology of  arbitrary shellable $q$-complexes are described in Section \ref{sec7}. 

\section{Preliminaries} \label{Sec2}

Throughout this paper $q$ denotes a power of a prime number and $\Fq$ the finite field with $q$ elements. We fix a positive integer $n$, and denote by $E$  the $n$-dimensional vector space $\Fq^n$ over $\Fq$. By $\SE$ we denote the set of all subspaces of $E$. 
Given any $y_1, \dots , y_r \in E$, we denote by $\<y_1, \dots , y_r \>$ the $\Fq$-linear subspace of $E$ generated by $y_1, \dots , y_r$. Also, for $U,V,W\in \SE$, we often write $U=V\oplus W$ to mean that $U = V+W$ and 
$V\cap W $ is the space $\{\0\}$ consisting of the zero vector in $E$. In other words, all direct sums considered in this paper are internal direct sums. 
 We denote by $\NN$ the set of all nonnegative integers, and by $\Np$ the set of all positive integers. 

Basic definitions and results concerning simplicial complexes and matroids will not be reviewed here. These are not formally needed, but they motivate the notions and results discussed below. If necessary, one can refer to \cite{Stan} or \cite{GSSV} for simplicial complexes, shellability, etc. and to \cite{White} for basics (and more) about matroids. 
 
 \begin{defn}\label{def:qcplx}
 By a $q$-complex on $E =\Fq^n$ we mean a subset $\Delta$ of $\SE$ satisfying the 
 property that for every $A \in \Delta$, all subspaces of $A$ are in $\Delta$.

Let $\Delta$ be a $q$-complex.  Elements of $\Delta$ are called \emph{faces} of $\Delta$. 
Faces of $\Delta$ that are maximal (w.r.t. inclusion) 
are called the \emph{facets} of $\Delta$. The \emph{dimension} of  $\Delta$ 
is  $ \max\{\dim A : A \in \Delta\}$, 
 and it is 
 denoted by $\dim \Delta$.  
We say that $\Delta$ is \emph{pure} if all its facets have the same dimension. 
  \end{defn}
  
\begin{exa} \begin{enumerate}
 \item[(i)]
 Clearly, $ \SE$ is 
 a pure $q$-complex of dimension $n$. Also,  $\Delta := \{ A\in \SE : A \ne E\}$ is a pure $q$-complex of dimension $n-1$; we denote it by $S_q^{n-1}$ and call it the \emph{$q$-sphere} of dimension $n-1$. 
  \item[(ii)] If $\AC$ is any subset of $\SE$, then $\{B \in \SE : B \subseteq A \text{ for some } A\in \AC\}$
  is a $q$-complex, called the \emph{$q$-complex generated by} $\AC$, and denoted by $\langle  \AC\rangle$. In case $\AC=\{A_1, \dots , A_r\}$, we often write $\langle  \AC\rangle$ as $\langle  A_1, \dots , A_r \rangle$. 
  By convention, if $\AC$ is the empty set, then $\langle  \AC\rangle$ is defined to be the empty set. 
 \end{enumerate}
\end{exa}

We now recall the definition of a $q$-matroid, as given by Jurrius and Pellikaan~\cite{JP}.

\begin{defn} 
\label{defn:7}
A \emph{$q$-matroid} on $E$  is a pair $M = (E,\r)$, 
where $\r: \SE \to \NN$ is a function (called the \emph{rank function} of $M$) satisfying the following properties. 
\begin{enumerate}[(r1)]
\item \label{r1}$0\leq \r(A) \leq \dim A$ for all $A\in \SE$, 
\item \label{r2} If $A, B \in \SE$ with $A \subseteq B$, then $\r(A)\leq \r(B)$,
\item \label{r3} $\r(A+B) + \r(A\cap B) \leq \r(A) + \r(B)$ for all $A, B\in \SE$. 
\end{enumerate}
\end{defn}

\begin{defn} 
Let $M = (E,\r)$ 
be a $q$-matroid. We call $\r(E)$ the \emph{rank} of $M$. 
Let $A \in \SE$. 
Then $A$ 
is said to be \emph{independent} (in $M$) if $\r(A) = \dim A$; otherwise it is called \emph{dependent}. 
 Further, 
$A $ is 
a \emph{basis} (of $M$) if $A$ is independent and $\r(A)= \r(E)$. 
  
\end{defn}

\begin{exa}\label{exa:uniform}
Given a positive integer $k\le n$, consider $\rho: \SE \to \NN$ defined by 
$$
\rho(A) = \begin{cases} \dim A & \text{if } \dim A \le k, \\  k  & \text{if } \dim A > k. \end{cases}
$$
Then it is easily seen that $(E, \rho)$ is a $q$-matroid of rank $k$; this is called the \emph{uniform $q$-matroid} on $E$ of rank $k$, and it is denoted by $U_q(k,n)$. 
\end{exa}

Important properties of  independent subspaces in a $q$-matroid (which, in fact, characterize a $q$-matroid) 
are proved in \cite[Thm.~8]{JP} and recalled 
below. 

\begin{pro}
\label{thm:14}
Let $M = (E,\r)$ 
be a $q$-matroid, and let $\I$ be the family of independent subspaces in $M$. 
Then $\I$  satisfies the following four properties:
\begin{enumerate}[({i}1)]
\item \label{thm:14-i}$\I\neq \emptyset$.
\item \label{thm:14-ii}$A\in \SE$ and $B\in \I$ with $A\subseteq B$ $\Rightarrow$ $A\in \I$. 
\item \label{thm:14-iii}$A,B\in \I$ with $\dim A > \dim B$ $\Rightarrow$ there is $\x\in A\backslash B$ such that $B + \<\x\> \in \I$.
\item \label{thm:14-iv}$A, B \in \SE$ and $I, J$ 
are maximal independent subspaces of $A, B$,  
respectively 
$\Rightarrow$ there is a maximal independent subspace 
$K$ of $A+B$ such that $K \subseteq I+J$. 
\end{enumerate}
\end{pro}

It is shown in \cite{JP} that if $\I$ is an arbitrary subset of $\SE$ satisfying (i1)--(i4), then there is a unique $q$-matroid $M_{\I} = ( E, \r^{ }_{\I})$ whose 
rank function $\r^{ }_{\I}$ is given by 

$$
\r^{ }_{\I}(A) = \max\{ \dim B\colon B\in \I, \ B \subseteq A \} \quad \text{for  } A \in \SE; 
$$ 
moreover, $\I$ is precisely the family of independent subspaces in $M_{\I}$.

We now recall some fundamental properties of bases of a $q$-matroid, which provide yet another characterization of $q$-matroids. For a proof, see \cite[Thm. 37]{JP}.

\begin{pro}
\label{pro:2}
The set $\B$ of bases of a $q$-matroid on $E$ 
satisfies the following. 
\begin{enumerate}[({b}1)] 
\item $\B \neq \emptyset$.
\item 
If $B_1,B_2\in \B$ are such that  $B_1 \subseteq B_2$, then 
$B_1=B_2$.
\item 
 If $B_1 , B_2 \in \B$ and  $C \in \SE$ satisfy 
 $B_1 \cap B_2 \subseteq C \subseteq B_2$ and  \hbox{$\dim B_1 = \dim C + 1$}, then 
 there is $x\in B_1$ such that $C + \<x\>\in \B$.
\item 
 If $A_1,A_2 \in \SE$ and if $I_j$ is a maximal element of $\{B\cap A_j : B \in \B\}$ (with respect to inclusion) for $j=1,2$, then there is a maximal element  $J$ of \hbox{$\{B\cap (A_1+A_2) : B \in \B\}$} such that $J \subseteq I_1+I_2$. 
\end{enumerate}
\end{pro}

The third property here 
is called the {\em basis exchange property}. It can be used together with (b1) and (b2) to deduce that 
any two bases of a $q$-matroid have the same dimension. See, for example, \cite[Prop. 40]{JP}. 

As a consequence of Proposition~\ref{pro:2}, we shall 
derive the following {\em dual basis exchange property}, which will be useful to us in the sequel. 

\begin{cor}
\label{cor:1}
Let $M = (E,\r)$ 
be a $q$-matroid. 
Let $B_1 , B_2$ be bases of $M$ with $B_1 \neq B_2$ and let $y\in B_2\backslash B_1$. 
Then there exist  
$U \in \SE$ and $x\in B_1\backslash B_2$ such that 
\begin{equation}\label{eq:dbep}
B_1 \cap B_2 \subseteq U, \quad B_1 = U \oplus \<x\>, \quad \text{and} \quad \text{$U\oplus \<y\>$ is a basis of $M$.}
\end{equation}
\end{cor}

\begin{proof}
	Let $r:= \r(M)$ and $s:= r - \dim B_1\cap B_2$. Note that $1\le s \le r$. We will use (finite) induction on $s$.
	If $s=1$, then $U:= B_1\cap B_2$ and any $x \in B_1\backslash B_2$ clearly satisfy \eqref{eq:dbep}. Now suppose $s>1$ and the result holds for smaller values of $s$. Then $\dim B_1\cap B_2\leq r-2$, and so we can find $A\in \SE$ 
	and $y'\in B_2 \setminus B_1$ such that 
	$$
	 B_1 \cap B_2 \subseteq A \subseteq B_2  \quad \text{and} \quad  B_2 = A \oplus \<y\> \oplus \<y'\>.
	$$
Let $C:= A \oplus \<y\>$. Clearly, 	$B_1 \cap B_2 \subseteq C \subseteq B_2$ and $\dim B_1 = \dim C + 1$. So by (b3) in  Proposition~\ref{pro:2}, there is $x'\in B_1\backslash B_2$ such that $C\oplus\<x'\>$ is a basis of $M$.
Let $B_2' := C\oplus\<x'\>$.  Then 
$y \in B'_2 \backslash B_1$ 
and  $\dim B_1\cap B'_2 > \dim B_1 \cap B_2$. 
Hence, by the induction hypothesis, there is $U\in \SE$ and $x \in B_1\setminus B'_2$ such that  
$$
B_1 \cap B'_2 \subseteq U, \quad B_1 = U \oplus \<x\>, \quad \text{and} \quad \text{$U\oplus \<y\>$ is a basis of $M$.}
$$
Now observe that  $B_1\cap B_2 \subseteq B_1\cap A \subseteq B_1\cap C \subseteq B_1\cap B'_2$.  
Consequently, $x\in B_1\backslash B_2$ and \eqref{eq:dbep} holds. This completes the proof. 
	\end{proof}

We end this section by noting that if $M = (E,\r)$ is a $q$-matroid on $E = \Fq^n$ with $\r(M) = r$, then it follows from Proposition~\ref{thm:14} that $M$ defines a $q$-complex $\Delta_M$ 
whose faces are precisely the independent subspaces of $M$, i.e., those $\Fq$-linear subspaces $F$ of $\Fq^n$ such that $\dim F = \r(F)$. 
Moreover, the facets of $\Delta_M$ are precisely the bases of $M$. 
We will refer to $\Delta_M$ as the \emph{$q$-complex associated to $M$}. 
Since any two bases of $M$ have the same dimension $r$, it is clear that $\Delta_M$ is pure of dimension $r$. 
By a \emph{$q$-matroid complex} on $E$, we shall mean the $q$-complex associated to a $q$-matroid on $E$. 
Following Jurrius and Pellikaan \cite{JP}, a nontrivial example of $q$-matroid complex is provided by the following. 

\begin{exa}\label{exa:rankmetricmatroids}
Let $C$ be a (vector) rank metric code of length $n$ over an extension $\Fqm$ of $\Fq$, i.e., let $C$ be an $\Fqm$-linear subspace of $\Fqm^n$. Suppose $\dim_{\Fqm} C = k$. Let $G$ be a generator matrix of $C$, i.e., a $k \times n $ matrix with entries in $\Fqm$ whose rows form a basis of $C$. Given an $\Fq$-linear subspace $A$ of $E=\Fq^n$ with $\dim A =r$, let $Y_A$ denote a generator matrix of $A$, i.e., a $r\times n$ matrix with entries in $\Fq$  whose rows form a basis of $A$, and let $\rho^{ }_C(A) : = \mathrm{rank}(GY_A^T)$, where $Y_A^T$ denotes the transpose of $Y_A$.  It is shown in \cite[\S~5]{JP} that $(E, \rho^{ }_C)$ is a $q$-matroid of rank $k$. Hence
$$
\Delta_C:= \{A \in \SE  :  \mathrm{rank}(GY_A^T) = \dim A\}
$$
is a pure $q$-complex of dimension $k = \dim C$.  
\end{exa}

\section{Tower Decompositions} \label{Sec4}

Suppose $\Delta$ is a pure  $q$-complex on $\Fq^n$ of dimension $r$. Then its facets are certain $r$-dimensional subspaces of $\Fq^n$ and \textit{a priori} it is not clear how they can be linearly ordered. In this section, we consider a variant of row reduced echelon forms, called tower decompositions, which will allow us to put a total order on such subspaces. 

Fix a positive integer $r \le n$ and let $\GrE$ denote the Grassmannian consisting of all $r$-dimensional subspaces of $E=\Fq^n$. Given any $U\in \GrE$,  
let $\M_U$ be a generator matrix of $U$ in row echelon form, i.e., let $\M_U$ be a $r\times n$ matrix in row echelon form whose row vectors form a basis of $U$. 
We denote by $u_r,u_{r-1},\dots,u_1$ the row vectors of $\M_U$ so that 
\[
\M_U = \begin{bmatrix}
 u_r\\
 \vdots\\
 u_1
      \end{bmatrix}
\] 
We define subspaces $U_1,\dots,U_r$ of  $E$ 
and subsets $\overline{U}_1, \dots, \overline{U}_r$ of $E \setminus \{\0\}$ by 
\begin{equation}\label{eq:UiUbari}
U_i : = \< u_1, \dots , u_i\> \quad \text{and} \quad \overline{U}_i := U_i\backslash U_{i-1} \quad \text{for } i=1, \dots , r,
\end{equation}
where, by convention $U_0 := \{\0\}$. Further, we define 
\[
\T(U):=(U_1,U_2,\cdots,U_r),
\]
and we shall refer to this as the \emph{tower decomposition} of $U$. Observe that although  $\M_U$ (or equivalently, the vectors $u_1, \dots , u_r$) need not be uniquely determined by $U$, the subspaces $U_i$ (and hence the subsets $ \overline{U}_i $) are uniquely determined by $U$. To see this, it suffices to note that there is a unique generator matrix of $U$, say $\M^*_U$, which is in reduced row echelon form, and it is easily seen that the corresponding subspace $U^*_i$ is equal to $U_i$ for  each $i=1, \dots  , r$. Thus, the tower decomposition $\T(U)$ of $U$ depends only on $U$. Moreover, it is obvious that $\T(U)$ determines $U$, since $U= U_r$. 
Note also that the set $U\setminus \{\0\}$ of nonzero elements of $U$ has the disjoint union decomposition  
\begin{equation}\label{eq:disjunion}
U\setminus \{\0\} = 
\coprod_{i=1}^r  \overline{U}_i.
\end{equation}

\begin{defn}
Given any nonzero vector $u\in \Fq^n$, the \emph{leading index} of $u$, denoted $p(u)$, is defined to be the least positive integer $i$ such that the $i$-th entry of $u$ is nonzero. Further, given a subset $S$ of $\Fq^n$, the \emph{profile} $p(S)$ of $S$ is defined to be  the union of the leading indices of all of its nonzero elements, i.e., 
$$
p(S) = \left\{ p(u)  :  u\in  S\backslash\{\0\} \right\}.
$$
Note that the profile of $S$ can be the empty set if $S$ contains no nonzero vector. 
\end{defn}

\begin{lem} \label{lem::4}
Let $U \in \GrE$ and let $u_r, \dots , u_1$ be the rows of a generator matrix $\M_U$ of $U$ in row echelon form.
Then $p(u_1)> \cdots > p(u_r)$.
Further, given any $i \in \{1, \dots , r\}$, if $U_i$, $\overline{U}_i $ are as in \eqref{eq:UiUbari}, then $p(\overline{U}_i)= \{p(u_i)\}$, and 
for any $u\in U\setminus \{ \0\}$, 
\[  
u\in \overline{U}_i \Longleftrightarrow p(u)= p(u_i).
\]
\end{lem}

\begin{proof}
Since $\M_U$ has rank $r$ and it is in row-echelon form, it is clear that $u_1, \dots , u_r$ are nonzero and $p(u_1)> \cdots > p(u_r)$. Now fix $i \in \{1, \dots , r\}$. Then 
$p(u_i) < p(u_j)$ for $1\le j < i$. 
Consequently, if $u= c_1u_1+ \cdots + c_iu_i$ for some $c_1, \dots , c_i \in \Fq$ with $c_i\ne \0$, then 
$p(u) = p(c_iu_i) = p(u_i)$. This shows that $p(\overline{U}_i)= \{p(u_i)\}$. The last assertion follows from this 
together with \eqref{eq:disjunion}. 
\end{proof}

Fix an arbitrary total order $\prec$ on $\Fq$ such that $0 \prec 1 \prec \alpha$  for all $\alpha \in \Fq\backslash \{0,1\}$. This extends lexicographically to a total order on $E$, 
which we also denote by~$\prec$. 
For $v,w\in E=\Fq^n$, we may write $v \preceq w$ if $v\prec w$ or $v=w$. 


\begin{lem}\label{lem:7}
Let $v,w$ be nonzero vectors in $E= \Fq^n$. If $p(v)<p(w)$, then $w\prec v$.
\end{lem}

\begin{proof}
Let  $i \in \{1,\ldots,n\}$ be such that $p(v) =i$. Suppose $p(w) > i$. Then the $j$th coordinate $w_j$ of $w$ is  $0$ for $1 \leq j \leq i$, whereas the $i$-th coordinate of $v$ is nonzero. Hence it is clear from the definition of $\prec$ that $w\prec v$.
%
\end{proof}


 In what follows, for any nonempty subset $S$ of $E= \Fq^n$, 
we denote by $\min S$ the least element of $S$ with respect to the total order $\prec$ on $E$ defined above. 
Some simple observations concerning this notion are recorded below for ease of reference.

\begin{lem}\label{lem:7_1}
Let $S$ be a  nonempty subset of  $E= \Fq^n$. 
\begin{enumerate}
\item[{\rm (i)}] If  $S$ is closed with respect to multiplication by nonzero scalars (for example, if $S = \overline{U}_i$ for some $i$, where $\overline{U}_i$ are as in \eqref{eq:UiUbari} for some subspace $U\in \GrE$),  then the first nonzero entry of the vector $\min S$ in $\Fq^n$ is necessarily $1$.
\item[{\rm (ii)}] If $S = U \setminus \{\mathbf{0}\}$ is the set of all nonzero vectors in some subspace $U\in \GrE$, then 
$\min S = \min \overline{U}_1$, where $\overline{U}_1$ is as in \eqref{eq:UiUbari}. 
\end{enumerate}
\end{lem}

\begin{proof}
The assertion in (i) is clear since $1 \prec \alpha$ for all nonzero $\alpha \in \Fq$. 
To prove (ii),  let $U \in \GrE$ and let 
$\overline{U}_i$, $1\le i \le r$, be as in \eqref{eq:UiUbari}. Write  $u:= \min  \overline{U}_1$. 
Then  for each $v\in \overline{U}_2 \cup \dots \cup \overline{U}_r$,  by Lemma~\ref{lem::4} we see that $p(u) > p(v)$, and hence  $u \prec v$, thanks to Lemma~\ref{lem:7}.  Thus, from \eqref{eq:disjunion}, we obtain $u = \min (U \setminus \{\mathbf{0}\})$, as desired.
\end{proof}

We are now ready to define a nice total order on $\GrE$.

\begin{defn}\label{defn:13}
Let $U, V \in \GrE$ and let $\T(U) =(U_1,\dots,U_r)$ and $\T(V)=(V_1,\dots,V_r)$ be the tower decompositions of $U$ and $V$, respectively. Define $U \preccurlyeq V$ if either $U=V$ or if there exists a positive integer $e \le r$ such that 
$$
U_j = V_j \text{ for } 1\le j < e , \quad U_e \ne V_e,  \quad \text{and} \quad \min  \overline{U}_e \prec \min \overline{V}_e.
$$
\end{defn}

\begin{lem} 
The relation 
$\preccurlyeq$  defined in Definition \ref{defn:13} is a total order on $\GrE$.
\end{lem}

\begin{proof}
Clearly, $\preccurlyeq$ is reflexive.  Next, let  $U, V \in \GrE$ and let 
$\T(U) =(U_1,\dots,U_r)$ and $\T(V)=(V_1,\dots,V_r)$ be their tower decompositions. 
If $U\ne V$, then there exists a unique positive integer $e \le r$ such that 
$U_j = V_j$  for  $1\le j < e $ and  $U_e \neq V_e$. Let $u: =  \min  \overline{U}_e $ and $v:=  \min  \overline{V}_e$. 
Observe that $U_e = U_{e-1} \oplus \< u\>$ and  $V_e = V_{e-1} \oplus \< v\>$. Since $U_{e-1} = V_{e-1}$ and 
$U_e \neq V_e$, it follows that $u\ne v$. Hence either  $u \prec v$ or  $v \prec u$. 
This shows that any two elements of $\GrE$ are comparable with respect to $\preccurlyeq$. 

It remains to show the transitivity of $\preccurlyeq$. To this end, suppose 
$U \preccurlyeq V$ and $V \preccurlyeq W$ for some $W \in \GrE$. Let $\T(W)=(W_1,\dots,W_r)$ be the tower decomposition of $W$. If $U=V$ or if $V=W$, then clearly $U\preccurlyeq W$. Suppose $U\ne V$ and $V\ne W$. Then there are unique 
integers $e, d \in \{1, \dots , r\}$ such that 
$$
U_j = V_j \text{ for } 1\le j < e , \quad U_e \ne V_e,  \quad \text{and} \quad \min  \overline{U}_e \prec \min \overline{V}_e.
$$
and
$$
V_j = W_j \text{ for } 1\le j < d , \quad V_d \ne W_d,  \quad \text{and} \quad \min  \overline{V}_d \prec \min \overline{W}_d.
$$
First, suppose $e<d$. Then it is clear that 
$$
U_j = V_j = W_j \text{ for } 1\le j < e , \quad U_e \ne V_e =W_e,  \quad \text{and} \quad \min  \overline{U}_e \prec \min \overline{V}_e = \min \overline{W}_e. 
$$
Hence $U\preccurlyeq W$. Likewise, if we suppose $d < e$, then 
$$
U_j = V_j = W_j \text{ for } 1\le j < d , \quad U_d = V_d  \ne W_d,  \quad \text{and} \quad \min  \overline{U}_d  = \min \overline{V}_d \prec \min \overline{W}_d. 
$$
So, we again obtain $U\preccurlyeq W$. Finally, if $e=d$, then the transitivity of $\prec$ on $E$ is readily seen to imply that $U\preccurlyeq W$. Thus $\preccurlyeq$ is a total order on $\GrE$.
    \end{proof}

 \section{Shellability of $q$-matroid complexes}   \label{Sec5}
 
 In this section, we begin with the definition of shellability of a $q$-complex and an equivalent formulation of it. Next, we shall use the results of the previous section to obtain a shelling of $q$-matroid complexes.

 The following definition is a straightforward analogue of the notion of shellability for $q$-complexes recalled in the Introduction. A slightly different, but obviously equivalent, definition was given by 
 Alder \cite[Definition 1.5.1]{Alder}. 

\begin{defn}\label{defn:11}
	Let $\Delta$ be a pure $q$-complex on $E=\Fq^n$. 
	A {\em shelling} of $\Delta$
is a linear order $F_1,\dots,F_t$  on the facets of $\Delta$ such that  for each $j=2, \dots , t$, the $q$-complex 
\hbox{$\langle F_j \rangle \cap \langle F_1, \dots F_{j-1} \rangle$} is generated by a nonempty set of maximal proper faces of $F_j$. 

We say that a $q$-complex is  \emph{shellable} if it is pure and it admits a shelling.
\end{defn}

\begin{exa} (Alder \cite[Example 1.5.2]{Alder}) \label{exa:qS}
A $q$-sphere $S_q^{n-1}$ is a shellable $q$-complex on $E=\Fq^n$ of dimension $n-1$. Indeed, its facets are the $(n-1)$-dimensional subspaces of $E$, and if $F_1, \dots , F_t$ is an arbitrary listing of these facets, then it is easily seen from the formula for the dimension of the sum of two subspaces, that $\dim (F_i \cap F_j) = n-2$ for $1\le i < j \le t$. Hence, for any $j=2, \dots , t$, we see that $\{F_i\cap F_j : 1\le i < j\}$  is a nonempty set of maximal proper faces of $F_j$, which generates $\langle F_j \rangle \cap \langle F_1, \dots F_{j-1} \rangle$. Thus
$F_1, \dots , F_t$ is a shelling of  $S_q^{n-1}$. 
\end{exa}

The following characterization is analogous to the corresponding result in the classical case (see, e.g., \cite[p. 135]{GSSV})  
and it will be useful to us in the sequel. 

\begin{lem}\label{lem:ShellChar}
Let $\Delta$ be a pure $q$-complex of dimension $r$, and let  $F_1, \dots , F_t$ be a listing of the facets of $\Delta$. Then $F_1, \dots , F_t$ is a shelling of $\Delta$  if and only if for every 
$i,j \in \Np$ with $ i<j\leq t$, there exists $k\in \Np$ with $k < j$ such that 
\begin{equation}\label{Fijk}
F_i \cap F_j \subseteq F_k \cap F_j \quad \text{and} \quad \dim(F_k \cap F_j) = r-1.
\end{equation}
\end{lem}

\begin{proof} 
Suppose $F_1, \dots , F_t$ is a shelling of $\Delta$. 
Let $i,j \in \Np$ with $ i<j\leq t$. 
Then $F_i \cap F_j \in \langle F_j \rangle \cap \langle F_1, \dots F_{j-1} \rangle$.
Hence 
$F_i \cap F_j \subseteq G_j$, where $G_j \in \langle F_j \rangle \cap \langle F_1, \dots F_{j-1} \rangle$ is a maximal proper face of $F_j$. Since $G_j \in \langle F_1, \dots F_{j-1} \rangle$,  there exists $k\in \Np$ with $k < j$ such that $G_j \subseteq F_k$. Thus, $G_j \subseteq F_k \cap F_j$ and moreover, 
$\dim G_j = \dim F_j -1 = r - 1$. Now, $F_k \ne F_j$, since $k< j$. Also, $\dim F_k = \dim F_j =r$. It follows that  
$\dim (F_k \cap F_j) \le r-1$. This implies that $G_j =  F_k \cap F_j$, and so \eqref{Fijk} is proved. 

Conversely, suppose for every $ i<j\leq t$, there exists 
$k < j$ such that \eqref{Fijk} holds. Let $j\in\{2, \dots, t\}$ and let $F$ be a face of $\langle F_j \rangle \cap \langle F_1, \dots F_{j-1} \rangle$. Then $F$ is a face of $F_j$ as well as $F_i$ for some $i<j$. For these $i, j$, there exists $k\in \Np$ with $k < j$ such that \eqref{Fijk} holds. Now $F \subseteq F_i \cap F_j \subseteq F_k \cap F_j $ and so $F$ is a face of $F_k \cap F_j $. It follows that $\{F_k \cap F_j : 1\le k < j \text{ and } \dim(F_k \cap F_j) = r-1\}$ constitutes a nonempty set of maximal proper faces, which generates $\langle F_j \rangle \cap \langle F_1, \dots F_{j-1} \rangle$. 
\end{proof}

We are now ready to prove the main result of this section.  Here we will make use of 
the total order $\preccurlyeq$ given in Definition~\ref{defn:13}. 
As usual, for any $U,V\in \SE$ of the same dimension, we will write $U \prec V$ to mean that $U \preccurlyeq V$ and $U\ne V$.

\begin{thm}\label{thm:shelling}
Let $M$ be a $q$-matroid on $E=\Fq^n$ of rank $r$. 
Then the $q$-complex $\Delta_M$ associated to $M$ is shellable. In fact, if  $F_1,\dots,F_t$ is an ordering of the facets of $\Delta_M$ such that $F_i\prec F_j$ for $1\le i<j \le t$, then this defines a shelling of $\Delta_M$.
\end{thm}
 
 \begin{proof}
 
We have seen already $\Delta_M$ is a pure $q$-complex of dimension $r$. Let  $F_1,\dots,F_t$ be an ordering of the facets of $\Delta_M$ such that $F_1 \prec \dots \prec F_t$. Fix integers $i, j$ with 
$1\le i<j \le t$. We need to show that there is a positive integer $k < j$ such that \eqref{Fijk} holds. 
This will be done in several steps. First, let us denote the tower decompositions of $F_i$ and $F_j$ by 
$$
\T(F_i)=(W_1,\dots,W_r) \quad \text{and} \quad \T(F_j)=(V_1,\dots,V_r).
$$
Since  $F_i\prec F_j$, there is a unique positive integer $e\le r$ such that 
$$
W_1=V_1, \dots , W_{e-1} = V_{e-1}, \quad W_e \ne V_e,  \quad \text{and} \quad   \min \overline{W}_e  \prec  \min \overline{V}_e.
$$
Write $w:= \min \overline{W}_e$ and $ v := \min \overline{V}_e$. We claim that $w\in F_i \setminus F_j$. 
Clearly, $w\in F_i$ and $w\ne 0$. Suppose if possible $w\in F_j$. 
Since $w \prec v$,  by Lemma \ref{lem:7}, we see that we can not have $p(v) > p(w)$. Thus, $p(v)\leq p(w)$. Further, if $p(v) = p(w)$, then by Lemma \ref{lem::4}, 
$p(\overline{V}_e) = \{p(v)\} = \{p(w)\}$, and since $w \in F_j \setminus \{\0\}$, it follows from  Lemma \ref{lem::4} that $w \in \overline{V}_e$. But this contradicts the minimality of $v$ in $\overline{V}_e$ since $w \prec v$. Thus 
$p(v) < p(w)$. Now $w \in F_j \setminus \{\0\}$ with $p(w) > p(v)$ and $p(\overline{V}_e) = \{p(v)\} $. Hence   it follows from  Lemma \ref{lem::4} that $w \in \overline{V}_s$ for some positive integer $s < e$. But then 
$w \in \overline{W}_s$ and so by Lemma \ref{lem::4}, $p(\overline{W}_s) = \{p(w)\} = p(\overline{W}_e)$, which is a contradiction. This proves the claim. 

Since $w\in F_i \setminus F_j$, we use the dual basis exchange property (Corollary~\ref{cor:1}) to obtain $U\in \SE$ and $x\in F_j \setminus F_i$ such that
$$
F_i \cap F_j \subseteq U, \quad F_j = U \oplus \< x\>, \quad \text{and} \quad U \oplus \< w\> \text{is a basis of } M.
$$
The last condition implies that $U \oplus \< w\> = F_k$ for a unique positive integer $k \le t$. Now it is clear 
that $F_i \cap F_j \subseteq U \subseteq F_k \cap F_j$. Further, if we show that $k < j$, then $F_k\cap F_j$ would be a proper subspace of $F_k$ and hence $\dim F_k \cap F_j \le r -1$. On the other hand, since $\dim U = r-1$ and 
$U \subseteq F_k \cap F_j$, we see that $\dim  F_k \cap F_j  = r-1$. 

To prove that $k< j$, we consider the tower decompositions  of $U$ and $F_k$, say, 
$$
\T(U) =(U_1,\dots,U_{r-1})  \quad \text{and} \quad   
\T(F_k) = (V^*_1, \dots , V^*_r).
$$
Recall that $W_s = V_s$ for $1\le s < e$. We now claim that $U_s = V_s$   for $1\le s < e$. To see this, let $d$ be the least positive integer such that $U_d \ne V_d$. Suppose, if possible $d < e$. Let $\alpha:= \min  \overline{U}_d$ and $\beta:=  \min \overline{V}_d$. Note that $\alpha \in U \setminus \{\0\} \subseteq F_j \setminus \{\0\}$. Now if $p(\alpha) = p(\beta)$, then from Lemma \ref{lem::4} we see that $\alpha \in \overline{V}_d$. 
Consequently, $V_d = V_{d-1} \oplus \< \alpha \> = U_{d-1} \oplus \< \alpha \> = U_d$, which is a contradiction. 
Also, if $p(\alpha) > p(\beta)$, then  from Lemma \ref{lem::4} we see that $\alpha \in \overline{V}_s$ for some positive integer $s < d$. But then $\alpha \in {V}_s = U_s \subseteq U_{d-1}$, which is a contradiction since 
$\alpha \in \overline{U}_d = U_d \setminus U_{d-1}$.  It follows that $p(\alpha) < p(\beta)$.  Finally, if $p(\alpha) <  p(\beta)$, then  from Lemma \ref{lem::4} we see that $\beta \in \overline{U}_s$ for some positive integer $s < d$. But then $\beta \in {U}_s = V_s \subseteq V_{d-1}$, which is a contradiction since 
$\beta \in \overline{V}_d = V_d \setminus V_{d-1}$. This proves that $d \ge e$ and so the last claim is proved.

Now let $\ell$ be the least positive integer such that $V_\ell \ne V^*_\ell$.  We shall show that $k < j$, or equivalently, $F_k \prec F_j$. by considering separately the following  
two cases. 

\smallskip

{\bf Case 1.} $\ell < e$. 

Let $v^{ }_{\ell}:= \min  \overline{V_{\ell}}$ and $v^*_{\ell}:= \min  \overline{V^*_{\ell}}$. Note that if 
$p(v^*_{\ell}) > p(v^{ }_{\ell})$, then by Lemma~\ref{lem:7}, $v^*_{\ell} \prec v^{ }_{\ell}$, and so $F_k \prec F_j$.
Thus, to complete the proof in this case it suffices to show that $p(v^*_{\ell}) \le p(v^{ }_{\ell})$ leads to a contradiction. 

First suppose $p(v^*_{\ell}) < p(v^{ }_{\ell})$. 
Since $\ell < e \le d$, we find $ v^{ }_{\ell} \in V_{\ell} = U_{\ell} \subseteq F_k$  and $v^{ }_{\ell} \ne 0$. Thus, from Lemma \ref{lem::4} we see that $ v^{ }_{\ell} \in V^*_s$ for some positive integer $s < \ell$. But then $V^*_s = V_s \subseteq V_{\ell -1}$ and so $v^{ }_{\ell} \in V_{\ell -1}$, which is a contradiction. 

Next, suppose $p(v^*_{\ell}) = p(v^{ }_{\ell})$. In this case, if $v^*_{\ell} \in F_j$, then we must have $v^*_{\ell} \in V_{\ell}$, thanks to Lemma \ref{lem::4}. But then $V^*_{\ell} = V^*_{\ell-1} \oplus \< v^*_{\ell} \> = V_{\ell}$, which is a contradiction. Thus $v^*_{\ell} \not \in F_j$. In particular, if 
$y:= v^*_{\ell} -  v^{ }_{\ell}$, then $y\ne 0$. Moreover, 
by part (i) of Lemma~\ref{lem:7_1}, 
the first nonzero entry in  $v^*_{\ell}$ as well as $ v^{ }_{\ell}$ is $1$. Hence $p(y) >p( v^*_{\ell}) =p(  v^{ }_{\ell})$. Also, 
$y\in F_k$, since $v^*_{\ell} \in F_k$ and  $ v^{ }_{\ell} \in V_{\ell} = U_{\ell} \subseteq F_k$. Thus, from Lemma \ref{lem::4}, we see that $y \in V^*_s$ for some positive integer $s < \ell$. But then $y \in V_s$, and so 
$y\in F_j$, which is a contradiction. This completes the proof in Case~1. 

\smallskip

{\bf Case 2.} $\ell \ge e$. 

Here $V^*_s = V_s = W_s$ for $1\le s < e$. Also $w \prec v$, where $w = \min \overline{W}_e$ and $ v = \min \overline{V}_e$. So by Lemma~\ref{lem:7}, $p(v) \le p(w)$. Now pick any $z\in  \overline{V^*_{e}}$ so that 
$V^*_e = V^*_{e-1} \oplus \< z\> = V_{e-1}  \oplus \< z\>$ and, by  Lemma \ref{lem::4}, $p(V^*_e) =\{p(z)\}$. 
Now $w\in F_k \setminus \{\0\}$ and so $w \in V^*_s$ for a unique positive integer $s \le r$. Also since 
$w \in \overline{W}_e$, we see that $w\not\in W_{e-1} = V^*_{e-1}$. Thus $s \ge e$ and therefore, in view of 
Lemma \ref{lem::4}, $p(v) \le p(w) \le p(z)$. Now if $p(v) = p(z)$, then $p(w) = p(z)$, and so $w \in \overline{V^*_e}$. Consequently, 
$$
\min \overline{V^*_e} \preceq w \prec v =  \min \overline{V}_e,
$$
which implies that $F_k \prec F_j$. 
On the other hand, if $p(v) < p(z)$, then by Lemma~\ref{lem:7}, 
$z \prec v$, and hence 
$$
\min \overline{V^*_e} \preceq z \prec v =  \min \overline{V}_e,
$$
which implies once again that $F_k \prec F_j$, as desired. 
\end{proof}

We remark that the shellability of the $q$-sphere $S_q^{n-1}$ is a trivial consequence of Theorem~\ref{thm:shelling}, because $S_q^{n-1}$ is precisely the $q$-matroid complex corresponding to the uniform $q$-matroid $U_q(n-1, n)$.


\section{Homology of $q$-Spheres and Uniform $q$-Complexes}  \label{Sec6}

This section is divided into three subsections. In \S~\ref{ss1} below, we review some preliminaries concerning finite topological spaces and their homotopy. Next, we consider $q$-spheres and explicitly determine their reduced homology groups in \S~\ref{ss2}. These results are then generalized in \S~\ref{newss3} to $q$-complexes associated to arbitrary uniform $q$-matroids. 


\subsection{Topological Preliminaries}\label{ss1}
Finite topological spaces, or in short, finite spaces, are simply topological spaces having only a finite number of points. In case they are $T_1$, the topology is necessarily discrete and not so interesting. Rather surprisingly, finite spaces that are $T_0$ (but not $T_1$) have a rich structure and a close connection with finite posets. The study of finite spaces goes back to Alexandroff \cite{Alex} and it has had important contributions by Stong \cite{Stong} and McCord \cite{McCord}. Good expositions of the theory of finite spaces are given by May \cite{May2012} and Barmak \cite{Bar11}. Still, the theory is not as widely known as it should, and so for the convenience of the reader, we provide here a quick review of the relevant notions and results. 

Let $X$ be a finite $T_0$ space. Then for each $x\in X$, the intersection, say $U_x$,  of all open sets of $X$ containing $x$ is open. Clearly $\{U_x : x\in X\}$ is a basis for (the topology on) $X$. For $x,y\in X$, define
$x\le y$ if $x \in U_y$. Then this defines a partial order on $X$ (since $X$ is $T_0$); moreover $U_y$ becomes the ``basic down-set'' $\{x\in X: x\le y\}$.

On the other hand, suppose $X$ is a finite poset (with the partial order denoted by $\le$). We call a subset $U$ of $X$ a \emph{down-set} (resp. \emph{up-set}) if whenever $y\in U$ and $x\in X$ satisfy $x\le y$ (resp. $y \le x$), we must have $x\in U$. We can define a topology on $X$ by declaring that the open sets in $X$ are precisely the down-sets in $X$ (or equivalently, the closed sets in $X$ are precisely the up-sets in $X$). This is called the \emph{order topology} on $X$, and it makes $X$ a finite $T_0$ space. 

Let $X,Y$ be finite posets, both regarded as finite topological spaces with the order topology. Then it can be shown (cf. \cite[Proposition 1.2.1]{Bar11}) that a function $f:X\to Y$ is continuous if and only if it is order-preserving. Further, if we 
let $Y^X$ denote the set of all continuous functions from $X$ to $Y$, then $Y^X$ is a poset with the pointwise partial order defined (for any $f, g \in Y^X$) by $f\le g$ if $f(x)\le g(x)$ for every $x\in X$. Thus  $Y^X$ can also be regarded as a finite topological space with the order topology. Moreover, $f,g\in Y^X$ are \emph{homotopic} 
(which means, as usual, that there is a continuous map $h:X\times [0,1] \to Y$ such that $h(x,0)=f(x)$ and 
$h(x,1)=g(x)$ for all $x\in X$) if and only if there is a continuous map $\alpha:[0,1]\to Y^X$ such that $\alpha(0) =f$ and $\alpha (1)=g$. We write $f\simeq g$ if $f,g \in Y^X$ are homotopic. Also, $X$ and $Y$ are said to be \emph{homotopy equivalent} if there are 
$f \in Y^X$ and $g \in X^Y$ such that  $f\circ g\simeq \Id_Y$ and $g\circ f\simeq \Id_X$.
Finally, recall that $X$ is said to be \emph{contractible} if it is homotopy equivalent to a point. Note that the homotopy groups as well as the reduced (singular) homology groups of contractible spaces are all trivial. 
Recall also that a topological space is \emph{acyclic} if all of its reduced homology groups are trivial. 
A contractible space is acyclic, but the converse is not true, in general. 

We now recall some known basic results  for which a reference is given. These will be useful to us later. Unless mentioned otherwise, the topology on finite posets is assumed to be the order topology and topological notions such as continuity, contractibility are considered with respect to this topology. 

\begin{pro}[{\cite[Corollary 1.2.6]{Bar11}}] 
\label{pro:fence}
Let $X, Y$  be finite posets 
and let $f,g\in Y^X$. 
Then  $f\simeq g$ if  and only if there is a finite sequence $f_0, f_1, \dots , f_t$ in $Y^X$ such that 
$f=f_0 \le f_1 \ge f_2 \le \cdots f_t = g$. 
\end{pro}

\begin{pro}[{\cite[Corollary 2.3.4]{May2012}}] 
\label{pro:minimal}
Let $X$ be a finite poset such that $X$ has a unique maximal element or a unique
minimal element. Then X is contractible.
\end{pro}

A finer version of Proposition~\ref{pro:minimal} for the posets that are of interest to us in this article is the following.  

\begin{lem}\label{pro:::6}
Let $\Delta$ be a nonempty collection of 
 subspaces of $E=\Fq^n$. 
 Call the elements of $\Delta$ as the faces of $\Delta$  and those faces of $\Delta$  that are maximal with respect to inclusion as the facets of $\Delta$.  Assume that any  finite intersection of facets of $\Delta$ that contains a fixed face of $\Delta$ is necessarily a face of $\Delta$. Suppose there is $A\in \Delta$ such that $A\subseteq F$ for every facet $F$ of $\Delta$. Then $\Delta$ is contractible.
\end{lem}

\begin{proof}
Fix any $B\in \Delta$. Consider $f\colon \Delta\rightarrow \{B\}$ and $g\colon \{B\} \rightarrow \Delta$ defined by 
$$
f(U)  := B  \text{ for all } U\in \Delta \quad \text{and} \quad g(B):=A.
$$
Clearly, $f$ and $g$ are continuous and $f\circ g = \Id_{\{B\}}$. 
We will show that $g\circ f\simeq \Id^{ }_\Delta$. To this end, define, for any $U\in \Delta$, the set $V_U$ to be the intersection of all facets of $\Delta$ containing $U$.
Let $h\colon \Delta\rightarrow \Delta$ be defined by  $h(U): = V^{ }_U$  for  $U\in \Delta$. Observe that if  $U_1, U_2\in \Delta$ with $U_1\subseteq U_2$, then  any facet of $\Delta$ containing $U_2$ must  contain $U_1$,  and therefore, $V^{ }_{U_1} \subseteq V^{ }_{U_2}$. Thus $h$ is order-preserving and hence it is continuous.  
By our hypothesis, $A \subseteq V^{ }_U$ for every $U\in \Delta$. Hence $g\circ f \le  h$. Also, since 
$U \subseteq V^{ }_U$ for any $U\in \Delta$, we obtain $\Id^{ }_\Delta\le h$. Thus it follows from 
by Proposition~\ref{pro:fence} that  $\Delta$ is homotopy equivalent to $\{B\}$. This proves that $\Delta$ is contractible.
\end{proof}

\begin{defn}\label{defn:::30}
A subset $\Delta$ of $\SE$  satisfying the hypothesis in Proposition \ref{pro:::6} is called a \emph{cone} with \emph{apex}  $A$.
\end{defn}

\subsection{Homology of $q$-Spheres}\label{ss2}

If $\Delta$ is a $q$-complex on $E=\Fq^n$, then $\Delta$ is a finite topological space with the order topology corresponding to the partial order given by inclusion. As a topological space, it is contractible because it has a unique minimal element, namely, the zero space $\{\mathbf{0}\}$ and so Proposition~\ref{pro:minimal} applies. Thus, the homology (as well as homotopy) groups of $\Delta$ are trivial. With this in view, and as in the classical case, we will replace 
$\Delta$ by the punctured $q$-complex
$$
\ring{\Delta}:= {\Delta} \setminus\left\{ \{\mathbf{0}\} \right\}
$$
obtained by removing the zero subspace from $\Delta$. Thus, when we speak of the homology of $\Delta$, we shall in fact mean the homology of $\ring{\Delta}$. In this subsection, we will outline how the (reduced) homology of $q$-spheres can be computed explicitly. 

Recall that the $q$-sphere ${S}_q^{n-1}$ is the $q$-complex formed by all the subspaces of $E=\Fq^n$ other than $E$ itself. So the punctured $q$-sphere $\ring{S}_q^{n-1}$ consists of all the subspaces of $E$
other than $E$ and $\{\mathbf{0}\}$. It is equipped with the order topology w.r.t. inclusion. 
In particular, $\ring{S}_q^{n-1}$ is the empty set if $n=1$. When $n=2$,  the punctured $q$-sphere $\ring{S}_q^{n-1}$ consists of  $q+1$ distinct one-dimensional subspaces of $\Fq^2$, which form connected components  with respect to the order topology. Thus the homology is rather easy to determine if $n=1$ or $n=2$. But the poset structure and the homology becomes a little more difficult to determine when $n\ge 3$. For example, 
the poset structure of the punctured $q$-sphere $\ring{S}_q^{2}$ when $q=2$ is depicted by (the solid lines in) Figure~\ref{fig:S2}, where we have let $x,y,z$ denote linearly independent elements of $\F_2^3$. It is seen here that  unlike in the case $n=2$, the $q$-sphere is a connected space when $n=3$. 

\begin{figure}\label{fig:S2}
\begin{center}
\hspace*{-.3in}
\begin{small}
\begin{tikzpicture}
  \node (max) at (0,2) {$\<x,y,z\>$};
  \node (xy) at (-6,0) {$\<x,y\>$};
  \node (xz) at (-4,0) {$\<x,z\>$};
  \node (yz) at (-2,0) {$\<y,z\>$};
  \node (xyz) at (0,0) {$\<x,y+z\>$};
  \node (yxz) at (2,0) {$\<y,x+z\>$};
  \node (xyxz) at (4.5,0) {$\<x+y,x+z\>$};
  \node (zxy) at (7,0) {$\<z,x+y\>$};
  \node (x) at (-6,-2) {$\<x\>$};
  \node (y) at (-4,-2) {$\<y\>$};
  \node (z) at (-2,-2) {$\<z\>$};
  \node (XY) at (0,-2) {$\<x+y\>$};
  \node (XZ) at (2,-2) {$\<x+z\>$};
  \node (YZ) at (4.5,-2) {$\<y+z\>$};
  \node (XYZ) at (7,-2) {$\<x+y+z\>$};
  \node (min) at (0,-4) {$\<0,0,0\>$};
  \draw[dash dot] (min) -- (x);
  \draw[dash dot] (min) -- (y);
  \draw[dash dot] (min) -- (z);
  \draw[dash dot] (min) -- (XY);
  \draw[dash dot] (min) -- (XZ);
  \draw[dash dot] (min) -- (YZ);
  \draw[dash dot] (min) -- (XYZ);
  \draw[dash dot] (max) -- (xy);
  \draw[dash dot] (max) -- (yz);
  \draw[dash dot] (max) -- (xz);
  \draw[dash dot] (max) -- (xyz);
  \draw[dash dot] (max) -- (yxz);
  \draw[dash dot] (max) -- (zxy);
  \draw[dash dot] (max) -- (xyxz); 
  \draw (xy) -- (x) -- (xz) -- (z) -- (yz) -- (y) -- (yxz) -- (XZ) -- (xyxz) -- (XY) -- (zxy) -- (XYZ) -- (xyz) -- (YZ)--(yz); 
  \draw (x) -- (xyz);
  \draw (y) -- (xy) -- (XY);
  \draw (z) -- (zxy);
  \draw (xz) -- (XZ); 
  \draw  (yxz) -- (XYZ);
  \draw   (xyxz)-- (YZ);
\end{tikzpicture}
\end{small}
\end{center}
\caption{Illustration of the punctured $q$-sphere $\ring{S}_q^{2}$ when $q=2$}
\end{figure}
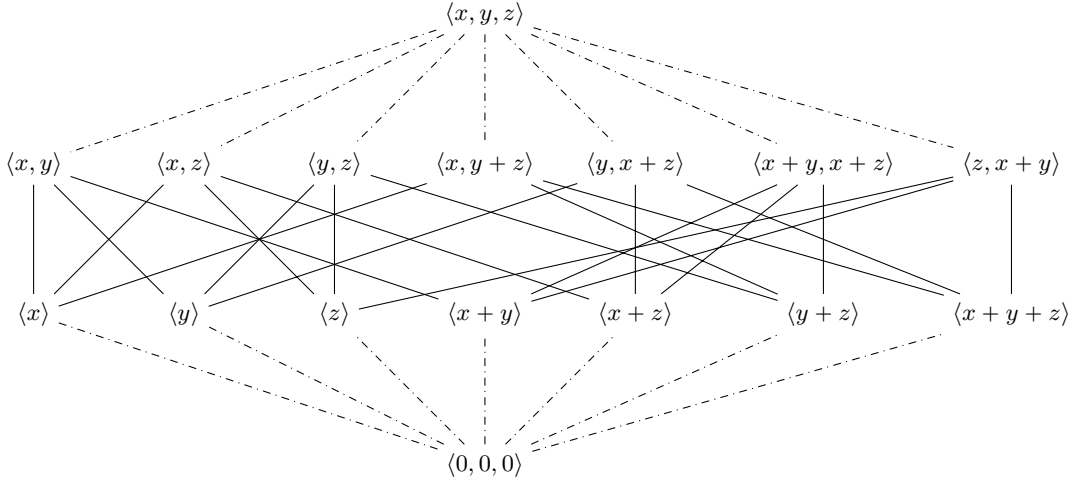

The key to determine the homology of $q$-spheres is the following lemma. Here, and hereafter, for an $\Fq$-vector space $F$, we denote by  $  \ring{\Sigma} (F)$ the set of all nonzero subspaces of $F$. 

\begin{lem} \label{lem:5.5}
Assume that $n \geq 2$. Then there exists a shelling $F_1,\ldots,F_t$ of the $q$-sphere $S_q^{n-1}$ and a positive integer $\ell \le t$ 
such that if for $1\le i \le t$, we let $\Delta_{i}:= \langle F_1,\ldots,F_{i} \rangle$,
then the punctured $q$-complex $\ring{\Delta}_{\ell}$ is contractible and moreover, 
\begin{equation}\label{DFi}
\ring{\Delta}_{i} \cap \ring{\Sigma} (F_{i+1})  =  \ring{\Sigma} (F_{i+1}) \setminus \{F_{i+1}\}
\quad \text{for } \ell \le i < t, 
\end{equation}
that is, $\ring{\Delta}_{i}\cap \ring{\Sigma} (F_{i+1}) $ is the punctured $q$-sphere $\ring{S}_{q}^{n-2}$ for each $i= \ell, \dots, t-1$. 
\end{lem}

\begin{proof}
We have seen in Example~\ref{exa:qS} that any ordering of the facets of $S_q^{n-1}$ gives a shelling of $S_q^{n-1}$. To obtain a   shelling with the additional two properties asserted in the lemma, 
we proceed as follows. 
Fix an arbitrary nonzero vector $\aa$ in $\Fq^{n}$. Suppose $F_1,\dots,F^{ }_{\ell}$  are all the facets of $S_q^{n-1}$ 
containing $\aa$. In other words $\{F_1,\dots,F^{ }_{\ell}\}$ is the set of all $(n-1)$-dimensional subspaces of $\Fq^n$,  
which contain $\aa$. Also, let $F^{ }_{\ell+1}, \dots , F^{ }_t$ denote all the facets of $S_q^{n-1}$, which do not contain $\aa$. Write $\Delta_{i}:= \langle F_1,\ldots,F_{i} \rangle$ for $1\le i \le t$. Then $\<\aa\>$ is contained in every facet  of $\ring{\Delta}_{\ell}$, and hence by Lemma~\ref{pro:::6},  $\ring{\Delta}_{\ell}$ is contractible. 

To prove that $F_1, \dots F^{ }_t$ also satisfies \eqref{DFi}, first suppose $n=2$. 
Then it is clear that $\ell =1 $ and $\ring{\Delta}_\ell = \{\<\aa\>\}$. Also, $ \ring{\Sigma} (F_{i+1})= \{F_{i+1} \}$ 
for $1\le i < t$. 
 Thus, we readily see 
 that the two sets on either sides of the equality in \eqref{DFi} are both empty. Now suppose $n\ge 3$. Fix $i\in \NN$ such that $ \ell \le i < t$. Since $F_{i+1}\not\in \Delta_i$, it is clear that 
$\ring{\Delta}_{i} \cap \ring{\Sigma} (F_{i+1}) \subseteq \ring{\Sigma} (F_{i+1}) \setminus \{F_{i+1}\} $. To prove the other inclusion, it suffices to show that every facet of $\Sigma (F_{i+1}) \setminus \{F_{i+1}\}$ is in $\Delta_i$. Let $G$ be a facet of $\Sigma (F_{i+1}) \setminus \{F_{i+1}\}$. Since $i\ge \ell$, we see that $\aa\not\in G$. Hence 
$G\oplus \<\aa\>$ is a facet of $S_q^{n-1}$ containing $\aa$, and therefore, $G\oplus \<\aa\> =F_k$ for some positive integer $k\le \ell$. In particular, $G\subseteq F_k$ and so $G\in \Delta_k \subseteq \Delta_i$. 
\end{proof}

\begin{rem}\label{rem:tl}
It is possible to describe the positive integers $t$ and $\ell$ in Lemma~\ref{lem:5.5} explicitly. Indeed, $t$ is the number of subspaces of $\Fq^n$ of dimension $n-1$. Also, the proof of Lemma~\ref{lem:5.5} shows that we can take $\ell$ to be the number of subspaces of $\Fq^n$ of dimension $n-1$ containing a fixed nonzero vector $\aa$. 
Consequently, both $t$ and $\ell$ can be described in terms of Gaussian binomial coefficients as follows. 
\[
t = {{n}\brack{n-1}}_q = \frac{q^n-1}{q-1} \quad \text{and} \quad \ell = {{n-1}\brack{n-2}}_q = \frac{q^{n-1}-1}{q-1}.
\]
Observe that $t - \ell = q^{n-1}$. 
\end{rem}

Let us recall that as per standard conventions in topology, if $X$ is the empty set, then its reduced homology group $\widetilde{H}_p(X)$ is $\ZZ$ if $p=-1$ and $0$ otherwise\footnote{Indeed, a $p$-simplex is the convex hull of $p+1$ points. So if $p=-1$, then this is the empty set, while the singular $p$-simplex in $X$ consists precisely of the empty function, and the free abelian group $C_p(X)$ generated by it is $\ZZ$. On the other hand, all other chain complexes are $0$.}. In general, the homology groups of (punctured) $q$-spheres are given by the following.

\begin{thm} \label{thm:5.7}
Let $c_n:= q^{n(n-1)/2}$. Then the reduced homology groups of the punctured $q$-sphere $\ring{S}_q^{n-1}$ are given by 
 \[
\widetilde{H_p}(\ring{S}_q^{n-1}) = \begin{cases}
\ZZ\strut^{c_n}  & \text{if } p=n-2,\\
0 & \text{otherwise } .
\end{cases}
\] 
\end{thm}
\begin{proof}
We use induction on $n$. If $n= 1$, then the desired result follows from the standard conventions about the reduced homology of the empty set.  

Now suppose $n \ge 2$ and the result holds for values of $n$ smaller than the given one. Let $F_1,\ldots,F_t$ be a shelling of ${S}_q^{n-1}$ as in Lemma~\ref{lem:5.5}, and let $\ell$ be the positive integer as in Lemma~\ref{lem:5.5} and Remark~\ref{rem:tl}. Also let $\Delta_i$, for $1\le i \le t$, be as in Lemma~\ref{lem:5.5}. In the first step, we take
$$
X_1:= \ring{\Delta}_{\ell} \quad \text{and} \quad X_2:= \ring{\Sigma}(F_{\ell+1}).
$$
Note that both $X_1$ and $X_2$ are down-sets, and thus they are open subsets of $\ring{S}_q^{n-1}$. Moreover, 
$X_1\cup X_2 =  \ring{\Delta}_{\ell +1}$, and by  Lemma~\ref{lem:5.5}, $X_1\cap X_2$ can be identified with the 
punctured $q$-sphere $\ring{S}_q^{n-2}$. Let us apply the Mayer-Vietoris sequence for reduced homology:
$$
\widetilde{H_p}(X_1)\oplus \widetilde{H_p}(X_2)\longrightarrow\widetilde{H_p}(X_1 \cup X_2) \longrightarrow \widetilde{H}_{p-1}(X_1 \cap X_2)\longrightarrow \widetilde{H}_{p-1}(X_1)\oplus \widetilde{H}_{p-1}(X_2)
$$
and observe that by  Lemma~\ref{lem:5.5},  $X_1$ is contractible, and since $X_2$ has a unique maximal element (viz., $F_{\ell+1}$), by Proposition~\ref{pro:minimal} , $X_2$ is also contractible. Thus both the direct sums in the above exact sequence are $0$, and we obtain
$$
  \widetilde{H_p}(\ring{\Delta}_{\ell +1} ) = \widetilde{H_p}(X_1 \cup X_2) \cong \widetilde{H}_{p-1}(X_1 \cap X_2) = \widetilde{H}_{p-1}(\ring{S}_q^{n-2}).  
$$
So by the induction hypothesis, $\widetilde{H_p}(\ring{\Delta}_{\ell +1} )$ is equal to $\ZZ\strut^{c_{n-1}}$ if $p-1 = n-3$, i.e., $p=n-2$, and $0$ otherwise. In the next step, we take 
$$
X_1:= \ring{\Delta}_{\ell +1} \quad \text{and} \quad X_2:= \ring{\Sigma}(F_{\ell+2}),
$$
and note that $X_1, X_2$ are open subsets of $\ring{S}_q^{n-1}$ such that $X_1\cup X_2 =  \ring{\Delta}_{\ell +2}$, and by  Lemma~\ref{lem:5.5}, $X_1\cap X_2$ can be identified with the 
punctured $q$-sphere $\ring{S}_q^{n-2}$. Let us apply (a slightly longer)  Mayer-Vietoris sequence for reduced homology:
$$
\minCDarrowwidth12pt
\begin{CD}
\widetilde{H_p}(X_1 \cap X_2) @>>> \! \widetilde{H_p}(X_1)\oplus \widetilde{H_p}(X_2) @>>>\!  \! \! \! \widetilde{H_p}(X_1 \cup X_2) \\
@. @. @VVV  \\\ 
@. @.  \!  \!  \! \! \widetilde{H}_{p-1}(X_1 \cap X_2) @>>> \widetilde{H}_{p-1}(X_1)\oplus \widetilde{H}_{p-1}(X_2)
\end{CD}
$$
This time $X_2$ is contractible, whereas the homology of $X_1$ is determined in the previous step, while that of 
$X_1\cap X_2$ is known, as before, by the induction hypothesis. Using this for $p=n-2$, we obtain
$$
0 \longrightarrow  \ZZ\strut^{c_{n-1}}  \longrightarrow
  \widetilde{H}_{n-2}(\ring{\Delta}_{l+2}) \longrightarrow  \ZZ\strut^{c_{n-1}} \longrightarrow 0.  
$$
The short exact sequence above splits (since $\ZZ\strut^{c_{n-1}}$ is a projective $\ZZ$-module, being free), and therefore  $ \widetilde{H}_{n-2}(\ring{\Delta}_{l+2}) = \ZZ\strut^{c_{n-1}} \oplus \ZZ\strut^{c_{n-1}} $. Moreover, 
$\widetilde{H}_{p}(\ring{\Delta}_{l+2}) = 0$ if $p\ne (n-2)$. 
Now if $\ell+2 < t$, we can proceed as before,  and we shall obtain that $\widetilde{H}_{p}(\ring{\Delta}_{l+3})$ is 
$ \ZZ\strut^{c_{n-1}} \oplus \ZZ\strut^{c_{n-1}}   \oplus \ZZ\strut^{c_{n-1}} $ if $p = n-2$, and $0$ otherwise. Continuing in this way, we see that $\widetilde{H}_{p}(\ring{\Delta}_{t})$ is the direct sum of $(t-\ell)$ copies of 
$\ZZ\strut^{c_{n-1}} $ if $p = n-2$, and $0$ otherwise. Now $\Delta_t = S_q^{n-1}$ and in view of Remark~\ref{rem:tl}, 
$$
(t - \ell) c_{n-1} = q^{n-1} q^{(n-1)(n-2)/2} = q^{n(n-1)/2} = q^{c_n}. 
$$
This yields the desired result. 
\end{proof}

It may be noted that Lemma \ref{lem:5.5} plays a crucial role in determining the homology of $q$-spheres. Indeed  
Theorem \ref{thm:5.7} can be readily extended to shellable $q$-complexes satisfying the hypothesis of Lemma \ref{lem:5.5}, and moreover the hypothesis in Lemma \ref{lem:5.5} about contractibility can be replaced by the slightly weaker hypothesis of acyclicity. We record this below. 

\begin{thm}\label{thm:12}
Let $\Delta$ be a pure $q$-complex  on $E=\Fq^n$ of positive dimension $d$. 
Assume that $F_1, \ldots, F_t$ is a shelling on $\Delta$ and there is 
$\ell \in \Np$ with $ \ell \leq t$ such that
 if for $1\le i \le t$, we let $\Delta_{i}:= \langle F_1,\ldots,F_{i} \rangle$,
then the punctured $q$-complex $\ring{\Delta}_{\ell}$ is acyclic and moreover, 
$\ring{\Delta}_{i} \cap \ring{\Sigma}(F_{i+1})$   is the punctured $q$-sphere $\ring{S}_q^{d-1}$ for $\ell  \le  i < t$. Then 
\[
\widetilde{H_p}(\ring{\Delta}) = \begin{cases}
\ZZ\strut^{(t- \ell) q^{d(d-1)/2}}
& \text{if } p=d-1,\\
0 & \text{otherwise.}
\end{cases}
\]  
\end{thm}

\begin{proof}
Follows using  similar arguments as in Theorem \ref{thm:5.7}.
\end{proof}

\subsection{Homology of Uniform $q$-Complexes}\label{newss3}

We shall now outline how the results of the previous subsection can be extended to the following more general class of $q$-complexes associated to arbitrary uniform $q$-matroids. 

\begin{defn}
Let $k$ be a nonnegative integer such that $k \leq n$. The \emph{uniform $q$-complex} of dimension $k$ is the 
$q$-complex  $\Delta_q(k,n)$  on $E=\Fq^n$  given by  
\[\Delta_q(k,n) := \{A \in \SE~:~ \dim A \leq k\}.
\]
\end{defn}
Note that $\Delta_q(k,n)$ is a pure $q$-complex  and its dimension is indeed $k$. Moreover, $\Delta_q(k,n)$ is precisely the $q$-matroid complex corresponding to the uniform $q$-matroid $U_q(k, n)$, and so it follows from 
Theorem~\ref{thm:shelling} that it is shellable. We shall now show that it admits a nice shelling, just as in the case of $q$-spheres.

\begin{lem}\label{lem:5.10}
Let $k$ be a positive integer such that $k \leq n$, and let $\Delta_q(k,n)$ be the uniform $q$-complex of dimension $k$. Then there exists a shelling $F_1, \ldots, F_t$ of $\Delta_q(k,n)$ and an integer $\ell$ with $1 \leq \ell \leq t$ such that  if for $1\le i \le t$, we let $\Delta_{i}:= \langle F_1,\ldots,F_{i} \rangle$, then 
 $\ring{\Delta}_{\ell}$ is contractible and 
$\ring{\Delta}_{i} \cap \ring{\Sigma}(F_{i+1})$   is the punctured $q$-sphere $\ring{S}_q^{k-1}$ for each $i= \ell, \dots, t-1$.  Moreover,  
$t = {{n}\brack{k}}_q$ and $\ell =  {{n-1}\brack{k-1}}_q$. 
\end{lem}

\begin{proof}
The facets of $\Delta_q(k,n)$ are precisely  the $k$-dimensional subspaces of $E=\Fq^n$. 
Consider the total order $\preccurlyeq$ on  $\GkE$ obtained using a total order $\prec$ on $E$ and tower decompositions as in 
Definition \ref{defn:13}. This induces a total order on the facets of $\Delta_q(k,n)$, which, by Theorem \ref{thm:shelling}, gives a shelling $F_1, \ldots, F_t$ of $\Delta_q(k,n)$, where 
$$
t = \text{ the number of $k$-dimensional subspaces of $E$} =  {{n}\brack{k}}_q.
$$

Let $\aa$ be the least nonzero element of $E$ with respect to the total order $\prec$. Note that if  $U, V$ are any two facets such that $\aa \in U$ and $\aa \notin V$, then in view of part (ii) of Lemma~\ref{lem:7_1}, we see that 
$\aa = \min \overline{U}_1 \prec \min \overline{V}_1$, and hence 
from Definition \ref{defn:13}, it follows that $U \prec V$. Now let 
$$
\ell = \text{ the number of $k$-dimensional subspaces of $E$ containing $\aa$} =  {{n-1}\brack{k-1}}_q
$$
so that the first $\ell$ facets $F_1, \ldots, F_\ell$ contain $\aa$, whereas the last $t-\ell$ facets $F_{\ell+1}, \dots , F_t$ do not contain $\aa$. Now $\ring{\Delta}_\ell $ is a cone with apex $\aa$, and hence by  Lemma \ref{pro:::6}, it is contractible. 

Next, let us 
fix an integer $i$ such that $\ell \le  i < t$. Since $F_{i+1}\not\in \Delta_i$, we see that 
 $\ring{\Sigma}(F_{i+1}) \cap \ring{\Delta}_{i}\subseteq \ring{\Sigma}(F_{i+1})\backslash \{F_{i+1}\}$. To prove the reverse inclusion,  it suffices to show that any subspace of $F_{i+1}$ of dimension $k-1$ is in $\Delta_{i}$. Let $G$ be a  subspace of $F_{i+1}$ with $\dim G = k-1$. Then  $\aa\notin G$, since $i > \ell$, 
and so 
$G\oplus \<\aa\> = F_j$ for some positive integer $j \le \ell$. In particular, $j \le i$ and  
$G\oplus \<\aa\> \in \Delta_i$.  
This implies that 
$G\in \Delta_{i}$. 
\end{proof}

We can now generalize  Theorem \ref{thm:5.7} from $q$-spheres to uniform $q$-complexes. 

\begin{thm}\label{thm:uniform}
Let $k \in \NN$ 
with $k \le n$, and let $c(n,k):= q^{k(k+1)/2} {{n-1} \brack {k}}_q$. 
 Then the reduced homology of  the uniform $q$-complex $\Delta_q(k,n)$ is given by 
\[
\widetilde{H_p}(\ring{\Delta}_q(k,n)) = \begin{cases}
\ZZ\strut^{c(n,k)} & \text{if } p=k-1,\\
0 & \text{otherwise}.
\end{cases}
\]  
\end{thm}

\begin{proof}
If $k=0$, then $c(n,k)=1$, while $\ring{\Delta}_q(k,n)$ is the empty set, and the result follows from standard conventions in topology. If $k$ is a positive integer $\le n$, then the result follows from Lemma~\ref{lem:5.10} and Theorem~\ref{thm:12} by noting that 
$$
t - \ell = {{n}\brack{k}}_q - {{n-1}\brack{k-1}}_q = q^k {{n-1}\brack{k}}_q \quad \text{and so}\quad 
(t - \ell) q^{k(k-1)/2} = c(n,k),
$$
where $t$ and $\ell$ are as in Lemma~\ref{lem:5.10}. 
\end{proof}

It may be remarked that Theorem~\ref{thm:uniform} is a trivial consequence of Proposition~\ref{pro:minimal}
when $k=n$, because in this case $\ring{\Delta}_q(k,n)$ 
contains a unique maximal element (viz.,  $E = \Fq^n$), while $c(n,k)=0$. 

\section{Homology of Shellable $q$-Complexes}\label{sec7}
We shall now attempt to determine the homology of a shellable $q$-complex. We proceed in a manner analogous to the classical case of simplicial complexes. But as we shall see, there are some difficulties in obtaining results
analogous to those in the classical case. 

\subsection{Intervals in Shellable $q$-Complexes}\label{ss6.1} 
In the classical case, the notion of \emph{restriction} $\R(F)$ of a facet $F$ plays an important role in the determination of the homology of a shellable simplicial complex; see, e.g., \cite[\S~7.2]{AB1}. But in the case of $q$-complexes, a straightforward analogue is not possible because the complement of an element (or even of a one-dimensional subspace) in an $\Fq$-linear subspace needs not be a subspace. Nonetheless, it turns out that we have a useful analogue if we consider a plethora of restrictions of a facet $F_j$ as defined below. The sets $I_j$ in this definition  provide an analogue of the intervals $[\R(F_j), F_j]$ in the classical case.


\begin{defn}\label{def:Rij}
Let $F_1,\dots,F_t$ be a shelling of a  shellable $q$-complex $\Delta$ on $E=\Fq^n$. 
For $1\le i < j \leq t$, 
the \emph{$i$th restriction} of $F_j$ is defined to be the set 
\[
\R_i(F_j) := \{ x\in F_j :  \langle x \rangle \oplus (F_i \cap F_j) = F_j \}.
\]
Further, for $1 \le j \le t$, we define
$$
I_j := \{ A \in \langle F_j \rangle \colon A\cap \R_{i}(F_j) \neq \emptyset \text{ whenever } 1\le i < j  \text{ and } \R_{i}(F_j)\neq \emptyset\}. 
$$
\end{defn}

\begin{rem}\label{rem:6.2}
If $i, j$ and $F_1, \dots, F_t$ are as in Definition~\ref{def:Rij} and if $F_i \cap F_j$ is not a hyperplane in $F_j$, i.e., if $\dim (F_i \cap F_j) < \dim F_j - 1$, then clearly $\R_{i}(F_j) = \emptyset$. On the other hand, 
 for each $j=2, \dots , t$, we can use Lemma~\ref{lem:ShellChar} to obtain $k\in \Np$ with $k < j$ such that $\R_k(F_j) \ne \emptyset$, and therefore, $I_j$ is nonempty. Note also that the defining condition for $I_j$ is vacuously true if $j=1$, and thus $I_1 = \langle F_1 \rangle$. In general, $\{F_j\} \subseteq I_j \subseteq \langle F_j \rangle$ for each $j=1, \dots , t$. 
\end{rem}

In the classical case, the interval $[\R(F_j), F_j]$ equals $\{F_j\}$ if and only if $\R(F_j) = F_j$. The following lemma is a partial  analogue of this in the case of $q$-complexes. 

\begin{lem}\label{lem:5.13}
Let $F_1,\dots,F_t$ be a shelling of a  shellable $q$-complex $\Delta$ on $E=\Fq^n$. If $j\in \{2, \dots , t\}$ is such that  $I_j = \{F_j \}$, then 
$$
\bigcup_{i=1}^{j-1} \R_{i}(F_j) = F_j \backslash \{\mathbf{0}\}.
$$
\end{lem}

\begin{proof}
Let $j\in \{2, \dots , t\}$ satisfy 
$I_j = \{F_j \}$. The inclusion $\cup_{i=1}^{j-1} \R_{i}(F_j) \subseteq F_j \backslash \{ \mathbf{0} \}$ is obvious. To prove the reverse inclusion, let $x\in F_j \backslash \{ \mathbf{0} \}$. Also let $A$ be a codimension $1$ subspace of $F_j$ such that $ \langle x \rangle \oplus A = F_j$. 
Since $j\ge 2$, 
 in view of Remark~\ref{rem:6.2}, there is $i\in \Np$ with $i< j$ such that $R_{i}(F_j) \ne \emptyset$. In particular, $\dim F_i \cap F_j = \dim F_j -1$. Now since 
  $I_j = \{F_j \}$, we see that $A \notin I_j$, and therefore $A \cap R_{i}(F_j) = \emptyset$. This implies that 
  $A \subseteq F_i \cap F_j$ and since $A$ has codimension $1$, we obtain $A = F_i \cap F_j$. Consequently, 
  $x\in R_{i}(F_j)$. 
%
%
%
\end{proof} 

Unlike in the classical case, the converse of Lemma \ref{lem:5.13} is not true, and this is shown by the following example\footnote{Some of the computations in this example are done using SageMath, and the code is available
upon request.}.

\begin{exa}\label{exa:6.4}
Consider the field extension $\mathbb{F}_{2^4}/\mathbb{F}_{2}$ of degree $4$, and let $a$ be a root in $\mathbb{F}_{2^4}$ of the irreducible polynomial $X^4 + X+ 1$ in $\mathbb{F}_{2}[X]$ so that $\mathbb{F}_{2^4} = \mathbb{F}_{2}(a)$. 
Let $C$ be the rank metric code of length 4 over the extension $\mathbb{F}_{2^4}$ of $\mathbb{F}_{2}$ such that a generator matrix of $C$ is given by 
\[ 
G:= 
\begin{pmatrix}
a^{2} + a + 1 \ & a^{2} & a^{3} + a + 1 \ & a^{3} + a^{2} + a + 1 \\
a^{2} + a + 1 & a^{3} + 1 \ & a & a + 1 \\
a^{2} + 1 & 1 & a^{2} + 1 & a^{3} + 1
\end{pmatrix}.
\]
Let $\Delta_C$ be the $q$-matroid complex on $\mathbb{F}_{2}^4$ associated to $C$ as in Example~\ref{exa:rankmetricmatroids}. Then 
$\dim \Delta_C = \mathop{\rm rank}(G)=3$. There are ${{4}\brack{3}}_2 = 15$ subspaces of $\mathbb{F}_{2}^4$ of dimension $3$ and it turns out that 14 among these are in $ \Delta_C$. In the shelling order of Definition \ref{defn:13}, these 14 facets of $ \Delta_C$, say $F_1, \dots , F_{14}$, can be explicitly listed as follows. 
$$
\begin{array}{l}
\langle \mathbf{e}_2, \, \mathbf{e}_3, \, \mathbf{e}_4 \rangle,  \ \,
\langle \mathbf{e}_1 + \mathbf{e}_2, \, \mathbf{e}_3, \,  \mathbf{e}_4 \rangle, \ \, 
\langle \mathbf{e}_1, \,  \mathbf{e}_2, \, \mathbf{e}_4 \rangle, \ \, 
\langle \mathbf{e}_1 + \mathbf{e}_3, \,  \mathbf{e}_2, \,  \mathbf{e}_4 \rangle, \ \, 
\langle \mathbf{e}_1,  \, \mathbf{e}_2 + \mathbf{e}_3, \, \mathbf{e}_4 \rangle, \\
\langle \mathbf{e}_1 + \mathbf{e}_3,  \, \mathbf{e}_2 + \mathbf{e}_3, \, \mathbf{e}_4 \rangle, \ \, 
\langle \mathbf{e}_1, \, \mathbf{e}_2, \, \mathbf{e}_3 \rangle,  \ \, 
\langle \mathbf{e}_1 + \mathbf{e}_4, \, \mathbf{e}_2, \,  \mathbf{e}_3 \rangle, \ \, 
\langle \mathbf{e}_1, \,  \mathbf{e}_2 + \mathbf{e}_4 , \, \mathbf{e}_3 \rangle, \\
\langle \mathbf{e}_1 + \mathbf{e}_4,  \,  \mathbf{e}_2 +  \mathbf{e}_4, \,  \mathbf{e}_3 \rangle, \ \, 
\langle \mathbf{e}_1 ,  \, \mathbf{e}_2 , \, \mathbf{e}_3 + \mathbf{e}_4 \rangle, \  \, 
\langle \mathbf{e}_1 + \mathbf{e}_4, \, \mathbf{e}_2, \, \mathbf{e}_3 + \mathbf{e}_4 \rangle, \\
\langle \mathbf{e}_1 , \,  \mathbf{e}_2 +  \mathbf{e}_3, \,    \mathbf{e}_3  + \mathbf{e}_4 \rangle, \ \, 
\langle \mathbf{e}_1  + \mathbf{e}_4,  \, \mathbf{e}_2 +  \mathbf{e}_4 , \,    \mathbf{e}_3  + \mathbf{e}_4 \rangle, 
\end{array}
$$
where for $1\le i \le 4$, by $\mathbf{e}_i$ we have denoted the element of $\mathbb{F}_{2}^4$ with $1$ in the $i$th position and $0$ elsewhere. We can take 
a generator matrix of $F_j$ to be 
the $3\times 4$ matrix $Y_j$, which has as its rows the elements of the given ordered basis of $F_j$, and it can be checked that the rank of  the $3 \times 3$ matrix $GY_j^T$ is indeed $3$ for each $j=1, \dots , 14$. Incidentally, 
 the only $3$-dimensional subspace of $\mathbb{F}_{2}^4$ missing in the above list is $F:=\langle \mathbf{e}_1, \, \mathbf{e}_3, \, \mathbf{e}_4 \rangle$ and 
 its generator matrix $Y$ has the property that 
 rank$(GY^T)=2$; indeed, 
$$
Y = \begin{pmatrix}
1 & 0 & 0 & 0 \\
0 & 0 & 1 & 0 \\
0 & 0 & 0 & 1
\end{pmatrix}
\quad
\text{and} \quad
(GY^T)
\begin{pmatrix}
1 \\  a^{3} + a^{2} + a + 1 \\  a^{2} + a 
\end{pmatrix}
= 
\begin{pmatrix}
0 \\ 0 \\ 0
\end{pmatrix}.
$$
Now let us consider the subspace $F_8 = \langle \mathbf{e}_1 + \mathbf{e}_4, \, \mathbf{e}_2, \,  \mathbf{e}_3 \rangle$ and its restrictions $ \R_{i}(F_8)$ for $1\le i < 8$. Observe that $F_1\cap F_8 = \langle \mathbf{e}_2, \,  \mathbf{e}_3 \rangle$ and hence by Definition~\ref{def:Rij}, 
$$
 \R_{1}(F_8) =\{    \mathbf{e}_1 + \mathbf{e}_4,  \; \mathbf{e}_1 + \mathbf{e}_3 + \mathbf{e}_4, \;  \mathbf{e}_1 + \mathbf{e}_2 + \mathbf{e}_4,  \; \mathbf{e}_1 + \mathbf{e}_2 + \mathbf{e}_3 + \mathbf{e}_4 \}
 $$
 Similarly, $F_2\cap F_8 = \langle \mathbf{e}_3, \,  \mathbf{e}_1 + \mathbf{e}_2 + \mathbf{e}_4\rangle$ and 
 $F_3\cap F_8 = \langle \mathbf{e}_1 + \mathbf{e}_4, \,  \mathbf{e}_2  \rangle$, and
 hence
\begin{eqnarray*}
 \R_{2}(F_8) & = & \{    \mathbf{e}_1 + \mathbf{e}_4,  \; \mathbf{e}_2 , \;  \mathbf{e}_1 + \mathbf{e}_3 + \mathbf{e}_4,  \;  \mathbf{e}_2 + \mathbf{e}_3 \} \text{ and } \\
  \R_{3}(F_8) &= & \{    \mathbf{e}_3 ,  \; \mathbf{e}_1 + \mathbf{e}_3 + \mathbf{e}_4, \;   \mathbf{e}_2 + \mathbf{e}_3,  \; \mathbf{e}_1 + \mathbf{e}_2 + \mathbf{e}_3 + \mathbf{e}_4 \}
\end{eqnarray*}
We see already that 
$$
\bigcup_{i=1}^{7} \R_{i}(F_8) = F_8 \backslash \{\mathbf{0}\}.
$$
We can also compute the remaining restrictions and these turn out to be as follows. 
\begin{eqnarray*}
  \R_{4}(F_8) &= & \{    \mathbf{e}_3 ,  \; \mathbf{e}_1 +  \mathbf{e}_4, \;   \mathbf{e}_2 + \mathbf{e}_3,  \; \mathbf{e}_1 + \mathbf{e}_2 + \mathbf{e}_4 \}, \\
   \R_{5}(F_8) & = & \{    \mathbf{e}_2 , \; \mathbf{e}_3 , \;  \mathbf{e}_1 + \mathbf{e}_2 + \mathbf{e}_4,  \;  \mathbf{e}_1 + \mathbf{e}_3 + \mathbf{e}_4 \}, \\
      \R_{6}(F_8) & = & \{    \mathbf{e}_2 , \; \mathbf{e}_3 , \;  \mathbf{e}_1 +  \mathbf{e}_4,  \;  \mathbf{e}_1 + \mathbf{e}_2 + \mathbf{e}_3 + \mathbf{e}_4 \},  \text{ and }  \\
         \R_{7}(F_8) & = & \{   \mathbf{e}_1 +  \mathbf{e}_4,  \;  \mathbf{e}_1 +  \mathbf{e}_2  + \mathbf{e}_4 , \;  \mathbf{e}_1 + \mathbf{e}_3 + \mathbf{e}_4,  \;  \mathbf{e}_1 + \mathbf{e}_2 +  \mathbf{e}_3 + \mathbf{e}_4 \}, 
\end{eqnarray*}
Considering these restrictions, it is clear that the interval $I_8$ corresponding to $F_8$ is 
$$
I_8 
= \{  \langle \mathbf{e}_1 + \mathbf{e}_4, \,  \mathbf{e}_3  \rangle, \; F_8 \}. 
$$
Thus $I_8 \ne \{F_8\}$ and so the converse of Lemma \ref{lem:5.13} is not true, in general. 
\end{exa}
It may be observed in the above example that $I_8 = \langle F_1, \dots , F_8 \rangle \setminus \langle F_1, \dots , F_7 \rangle$. This turns out to be a special case of a general phenomenon. 
In fact,  we have the following result, which may be regarded as a  $q$-analogue of  \cite[Proposition 7.2.2]{AB1}. 

\begin{thm}\label{thm:::5}
Let $F_1,\dots,F_t$ be a shelling of a shellable $q$-complex $\Delta$  on $E=\Fq^n$.
For any $j \in \NN$ with $ j \le t$, 
let $\Delta_j$ denote the subcomplex   $\langle F_1, \dots , F_j \rangle$ of $\Delta$ generated by $F_1, \dots , F_j$ (in particular, $\Delta_0= \emptyset$, as per our convention). Then 
\begin{equation}\label{eq:decomp}
\Delta_j  = I_j\cup \Delta_{j-1} \quad \text{and} \quad I_j \cap \Delta_{j-1}=\emptyset. 
\end{equation}
Consequently,  
we obtain a partition of $\Delta$ as a disjoint union of ``intervals'':
\begin{equation}\label{eq:partition}
\Delta = \coprod_{j=1}^t I_j.  
\end{equation}
\end{thm}

\begin{proof}
As noted in Remark~\ref{rem:6.2}. $I_1=\Delta_1$, and so  \eqref{eq:decomp} holds when $j=1$.
Now suppose $2\le j \le t$. The inclusion $I_j\cup \Delta_{j-1} \subseteq \Delta_j$ is obvious. To prove the other inclusion, suppose, on the contrary, there is $A\in \Delta_j$ such that $A\notin  I_j$ and $A\notin \Delta_{j-1}$. Then $A\subseteq F_j$. Moreover, there is $i\in \Np$ with 
$i < j$ such that $\R_{i}(F_j) \neq \emptyset$ and  $\R_{i}(F_j)\cap A=\emptyset$. Now, if $A\nsubseteq F_i\cap F_j$, then $A$ would contain an element of $\R_{i}(F_j)$, which is a 
contradiction. Thus, $A\subseteq F_i\cap F_j$, and therefore $A\in \Delta_{j-1}$, which is again a contradiction. 
This proves that 
$\Delta_j \subseteq I_j \cup \Delta_{j-1}$. Thus $\Delta_j  = I_j \cup \Delta_{j-1}$. 

Next, suppose there is $A\in I_j\cap \Delta_{j-1}$. 
Let $S:=\{i \in \Np: i < j\text{ and } \R_{i}(F_j) \neq \emptyset\}$.  Then $A\cap \R_{i}(F_j) \neq \emptyset$ for all $i \in S$,  and so we can choose $x_i \in A\cap \R_{i}(F_j)$ for each $i\in S$. Define $G:= \< \{ x_i \colon i\in S \}\>$. Now $G\in I_j$ and $G \subseteq  A \subseteq F_k$ for some $k<j$ (because $A\in \Delta_{j-1}$). Thus  $G\subseteq F_k\cap F_j$. By Lemma~\ref{lem:ShellChar},  there exists $\ell <j$ such that $ F_k\cap F_j \subseteq F_\ell\cap F_j$ and $\dim (F_\ell\cap F_j) = \dim F_j - 1$. Consequently, $\R_{\ell}(F_j)\ne \emptyset$, and so $\ell\in S$. But then 
$\langle x_\ell \rangle \oplus (F_\ell\cap F_j) = F_j$ (by the definition of $\R_{\ell}(F_j)$), which is a contradiction because 
$x_\ell\in G\subseteq F_\ell\cap F_j$.
This shows that $I_j \cap \Delta_{j-1}=\emptyset$ and thus  \eqref{eq:decomp} is proved. 

Finally,  \eqref{eq:partition} follows from  \eqref{eq:decomp} by noting that $\Delta = \Delta_t$ and $\Delta_1=I_1$.%
\end{proof}

\subsection{Acyclic Subcomplexes of Shellable $q$-Complexes}\label{ss-6.2}

Recall that for a finite dimensional vector space $F$ over $\Fq$, we use $\ring{\Sigma}(F)$ to denote the punctured $q$-complex formed by all the nonzero subspaces of $F$.

\begin{lem}\label{lem:4}
Let $F$ be a vector space of dimension $r$ over $\Fq$. Let $m \in \Np$  and let $G_1, \dots , G_{m}$
 be subspaces of $F$ of  dimension $r-1$. For $s \in \Np$ with $s \leq m$, define 
$$
 U_s := \{ x\in F\colon \<x\>\oplus G_s = F \} \quad \text{and} \quad 
 I := \{ A \in  \ring{\Sigma}(F) :  A\cap U_s \neq \emptyset \text{ for } s =1,\dots, m \}. 
$$
 Then 
 $$
 \ring{\Sigma}(F)\backslash I = \bigcup_{s=1}^m \ring{\Sigma}(G_s).
 $$
\end{lem}

\begin{proof}
Suppose $A\in \ring{\Sigma}(F)\backslash I$. Then $A\cap U_{s} = \emptyset$ for some $s\in \Np$ with $ s\leq m$. We claim that $A \subseteq G_{s}$. 
Indeed, if there is $x\in A\setminus G_{s}$, then 
$\<x\>\oplus G_{s} = F$. But then $x\in A\cap U_{s}$ which is a contradiction. Therefore $A\in \ring{\Sigma}(G_{s})$.

On the other hand, if $A$ is a nonzero subspace of $G_s$ for some $s\in \{1, \dots , m\}$, then 
any element $x$ of $A$ cannot be in $U_s$ because $\<x\>+G_s = G_s$. Thus $A\cap U_s = \emptyset$. Hence $A\notin I$. This proves the lemma. 
\end{proof}

The above lemma says that ${\Sigma}(F)\backslash I$ is a pure $q$-complex with facets $G_1, \dots, G_m$.
We show below that the corresponding punctured $q$-complex is particularly nice.

\begin{cor}\label{cor:contra}
Let the notations and hypothesis be as in Lemma \ref{lem:4}. Further let $U:=U_1 \cup \dots \cup U_{m}$. 
If $U \neq (F\backslash \{\0\})$ and if $x$ is any nonzero element of $F\setminus U$, then 
$\ring{\Sigma}(F)\backslash I$ is a cone with apex $x$. Consequently, $\ring{\Sigma}(F)\backslash I$ is contractible. 
\end{cor}

\begin{proof}
Suppose $U \neq (F\backslash \{\0\})$ and $x$ is any nonzero element of $F\setminus U$. We claim that $x\in G_s$ for every $s\in \{1, \dots , m\}$. To see this, suppose $x \in F\setminus G_s$ for some $s\in \{1, \dots , m\}$. Then $\<x\>\oplus G_s = F$, and so $x\in U_s$. But this is a contradiction, since $x\notin U$. Thus the claim is proved. Consequently, 
in view of Lemma \ref{lem:4}, we see that $\<x\>$ is contained in every facet of $\ring{\Sigma}(F)\backslash I$. Thus $\ring{\Sigma}(F)\backslash I$ is a cone with apex $x$.
The last assertion follows from Lemma~\ref{pro:::6}. 
\end{proof}

\begin{cor}\label{cor:3}
Let $F_1,\dots,F_t$ be a shelling of a shellable $q$-complex $\Delta$  on $E=\Fq^n$. Suppose there is  $j \in \Np$ with $2\leq j \leq t$ such that 
\begin{equation}\label{RijneFj}
\bigcup_{i=1}^{j-1} \R_{i}(F_j) \neq F_j \backslash \{ \mathbf{0}  \} . 
\end{equation}
Then $\ring{\Sigma}(F_j)\backslash I_j$ is contractible.
\end{cor}

\begin{proof} 
If   in Lemma~\ref{lem:4}, we take
$$
F = F_j \quad \text{and} \quad \{G_1, \dots , G_m\} = \{F_i\cap F_j : 1\le i < j \text{ and } \R_i(F_j) \ne \emptyset\},
$$
then we see that $G_1, \dots, G_m$ are subspaces of $F$ of codimension $1$, 
and moreover, 
$U = \cup_{i=1}^{j-1} \R_i(F_j)$ and 
$I = I_j$. Thus the desired result 
follows from Corollary \ref{cor:contra}.
\end{proof}

The following result can be viewed as an analogue for $q$-complexes of 
Bj\" orner's Acyclicity Lemma \cite[Lemma 7.7.1]{AB1} for shellable simplicial complexes. 

\begin{thm}\label{thm::9}
Suppose  $F_1,\dots,F_\ell$ is a shelling of a shellable $q$-complex $\Delta'$  on $E$ of positive dimension $d$, and let 
$\Delta_j := \< F_1, \dots , F_j \>$ for $1\le j \le \ell$. Assume that 
\eqref{RijneFj} holds for each $j = 2, \dots , \ell$.
Then $\ring{\Delta}'$ is acyclic. 
\end{thm}

\begin{proof}
We prove by induction on $i$ ($1\le i \le \ell$) that each $\ring{\Delta}_i$ is acyclic.
Notice that each $\Delta_i$ is shellable.
Since $\ring{\Delta}_1=\ring{\Sigma}(F_1)$, has a unique maximal element, by Lemma~\ref{pro:::6} we see that it is contractible, and therefore acyclic. Now assume that $1< j \le \ell$ and $\ring{\Delta}_{j-1}$ is acyclic. 
We want to show that $\ring{\Delta}_{j}$ is also acyclic. 
Note that  $\ring{\Sigma}(F_{j})$ is contractible, and hence acyclic, while $\ring{\Delta}_{j-1}$ is acyclic by the induction hypothesis.  
Moreover, by Theorem~\ref{thm:::5}, $\Delta_{j-1} = \Delta_j  \setminus I_j$, and by taking intersections  with $\ring{\Sigma}(F_{j}) $, we obtain 
$\ring{\Delta}_{j-1}\cap \ring{\Sigma}(F_{j}) = \ring{\Sigma}(F_{j})\backslash I_{j}$. So by Corollary \ref{cor:3}, it follows that $\ring{\Delta}_{j-1}\cap \ring{\Sigma}(F_{j}) $ is contractible.
%
Hence, by applying a Mayer-Vietoris sequence, we see that $\ring{\Delta}_{j} = \ring{\Delta}_{j-1} \cup \ring{\Sigma}(F_{j})$ is acyclic.
This completes the proof. 
\end{proof}

\subsection{Computation of Homology of Shellable $q$-Complexes}\label{ss-6.3}
It may be pertinent to 
begin by recalling 
how one determines the homology in the classical case of a shellable simplicial complex, say $\Delta$. The first step is to observe that the subcomplex $\Delta'$ generated by the facets $F$ of $\Delta$ with $\R(F) \ne F$ is acyclic. In the next step   we attach to $\Delta'$ a facet $F$ of $\Delta$ with 
$\R(F) = F$ and use the Mayer-Vietoris sequence to determine the homology of $\Delta'\cup \langle F\rangle$, and then use an inductive argument. 
See, for example, \cite[\S~7.7]{AB1} or \cite[pp. 138--139]{GSSV}. 
This approach works because the intersection $\Delta'\cap \langle F\rangle$ is the boundary complex 
 of $F$. And this boundary complex being a sphere, we know its homology.

Now let us turn to a shellable $q$-complex $\Delta$  on $E=\Fq^n$. We can similarly consider the subcomplex $\Delta'$ consisting of the facets $F_j$ for which  \eqref{RijneFj} holds. Then  Theorem \ref{thm::9} would imply that $\Delta'$ is acyclic, provided  the ordering of facets of $\Delta$ restricted on the facets of $\Delta^{\prime}$ gives a shelling of $\Delta^{\prime}$. 
Next, if we were to attach to $\Delta^{\prime}$ a facet $F =F_j$ for which  \eqref{RijneFj} does not hold, then 
 we do not know whether or not the intersection $\Delta^{\prime} \cap \ring{\Sigma}(F_i)$ is a (punctured) $q$-sphere. But if one could overcome these difficulties, then the homology can certainly be computed as shown by the following result, where we have allowed ourselves a generous hypothesis. 
 
 \begin{thm}\label{thm:5.18}
Let $\Delta$ be a pure $q$-complex on $E=\Fq^n$ of positive dimension $d$ such that $\Delta$ admits a shelling $F_1, \ldots, F_t$. Let 
$\Delta^{\prime} :=\<F_j : j \in J'\>$, where 
\begin{equation}\label{eq:JDelta}
J:=\! \bigg\{ j\in \{2, \dots, t\} : \bigcup_{i=1}^{j-1} \R_{i}(F_j) = F_j \backslash \{ \mathbf{0}  \} \bigg\} \quad \text{and}  \quad 
J':=  \{1, \dots, t\} \setminus J. 
\end{equation}
Assume that the ordering $F_1, \ldots, F_t$ 
restricted on the facets of $\Delta^{\prime}$ gives a shelling of $\Delta^{\prime}$ and 
that 
$\ring{\Sigma}(F_j) \cap \ring{\Delta}^{\prime}$ is the punctured $q$-sphere $\ring{S}_q^{d-1}$ for each $j \in J$.  Then 
\[
\widetilde{H_p}(\ring{\Delta}) = \begin{cases}
\ZZ\strut^{|J|q^{d(d-1)/2}}& \text{if } p=d-1,\\
0 & \text{otherwise.} 
\end{cases}
\] 
\end{thm}

\begin{proof}
The facets of $\Delta^{\prime}$ are $F_j$ as $j$ varies over $J'$, and  the ordering of these induced by the linear ordering $F_1, \dots , F_t$ is 
a shelling of $\Delta^{\prime}$. Moreover,   for $2\le j\le t$,
$$
j\in J' \Longrightarrow  \bigcup_{i=1}^{j-1} \R_{i}(F_j) \neq F_j \backslash \{ \mathbf{0}  \}  \Longrightarrow 
\mathop{\bigcup_{1\le i < j}}_{i \in J'} \R_{i}(F_j) \neq F_j \backslash \{ \mathbf{0}  \},
$$
because $ \R_{i}(F_j) \subseteq F_j \backslash \{ \mathbf{0} \}$ for all $i\ne j$. 
Hence it 
follows from 
Theorem \ref{thm::9} that $\Delta^{\prime}$ is acyclic.
%
Now using 
the second assumption together with suitable Mayer-Vietoris sequences and proceeding as in the proof of Theorem \ref{thm:5.7}, we obtain the desired result about the reduced homology groups of $\ring{\Delta}$. 
\end{proof}
 
 The above result explicitly determines the singular homology of arbitrary shellable $q$-complexes, provided the hypothesis of Theorem \ref{thm:5.18} is satisfied. We show below that this hypothesis is satisfied
by shellable $q$-complexes for which the converse of Lemma \ref{lem:5.13} is true.
%
 
\begin{pro}
Let $\Delta$ be a pure $q$-complex on $E=\Fq^n$ of positive dimension $d$ such that $\Delta$ admits a shelling $F_1, \ldots, F_t$. 
Suppose for any $j \in \{2, \dots , t\}$,
 $$
\bigcup_{1\le i <j} \R_{i}(F_j) = F_j \backslash \{\mathbf{0}\} \Longrightarrow   I_j = \{F_j \} .
$$
Also, let $J$ and $J'$ be as in \eqref{eq:JDelta}. 
Then $\Delta' :=\<F_j : j \in J'\>$ satisfies the following. 
\begin{enumerate}
\item[{\rm (i)}]
The ordering $F_1, \ldots, F_t$ 
restricted on the facets of $\Delta^{\prime}$ gives a shelling of $\Delta^{\prime}$.
\item[{\rm (ii)}]
$\ring{\Sigma}(F_j) \cap \ring{\Delta}^{\prime}$ is the punctured $q$-sphere $\ring{S}_q^{d-1}$ for each $j \in J$. 
\end{enumerate}
\end{pro}

\begin{proof}
For $1\le j \le t$, let $\Delta_j:= \< F_1, \dots , F_j \>$. 
Note that the facets of $\Delta'$ are given by $F_j$, where $j$ varies over $J'$. Evidently, $\Delta'$ is a pure complex on $E$ of dimension $d$. 
To show that it is shellable, let $i, j\in J'$ with $i< j$. By Lemma~\ref{lem:ShellChar} (applied to $\Delta$),  there is $k_1\in \{1, \dots , t\}$ with $k_1< j$ such that $F_i \cap F_j \subseteq F_{k_1} \cap F_j$ and 
$\dim F_{k_1} \cap F_j = d-1$. If $k_1 \in J'$, then we are done. If not, then $k_1 \in J$ and in particular, $k_1 \ge 2$. By our hypothesis, $I_{k_1} =\{F_{k_1}\}$.  Hence by Theorem~\ref{thm:::5}, $\Delta_{k_1} \setminus \Delta_{k_1-1} = \{F_{k_1}\}$. Consequently, $F_{k_1} \cap F_j \subseteq F_{k_2}$ for some $k_2\in \Np$ with $k_2 < k_1< j$. This implies that 
$F_{k_1} \cap F_j \subseteq F_{k_2} \cap F_j$, and since $\dim (F_{k_1}\cap F_j) = d-1$, we obtain $F_{k_1} \cap F_j = F_{k_2} \cap F_j$. Again, if $k_2\in J'$, then we are done. Or else, $k_2\in J$, and we can proceed as before to obtain $k_3\in \Np$ with $k_3 < k_2 < k_1< j$ and $F_{k_3} \cap F_j = F_{k_2} \cap F_j$. Since $k_1, k_2, k_3, \dots $ are positive integers, this process can not continue indefinitely. Hence there is $k\in J'$ with $k< j$ such that 
$F_i \cap F_j \subseteq F_{k} \cap F_j$ and 
$\dim (F_{k} \cap F_j) = d-1$. This proves that $\Delta'$ is shellable and the ordering $F_1, \ldots, F_t$ 
restricted on the facets of $\Delta^{\prime}$ gives a shelling of $\Delta^{\prime}$. Thus (i) is proved. 

Next, let $j\in J$. Then $j\ge 2$ and $F_j \notin \Delta'$. Hence $\ring{\Sigma}(F_j) \cap \ring{\Delta}^{\prime} \subseteq \ring{\Sigma}(F_j) \setminus \{F_j\}$. We claim that the reverse inclusion also holds, i.e., 
$\ring{\Sigma}(F_j) \setminus \{F_j\} \subseteq \ring{\Sigma}(F_j) \cap \ring{\Delta}^{\prime}$. This is trivial if $d=1$. So we may assume that $d\ge 2$. 
Let $F$ be a facet of $\ring{\Sigma}(F_j) \setminus \{F_j\} $, i.e., a nonzero subspace of $F_j$ with $\dim F = d-1$.
Then $F\in \Delta_j$ and since $j\in J$, by Theorem~\ref{thm:::5} and our hypothesis, we see that 
$F\notin \{F_j\} = \Delta_j \setminus \Delta_{j-1}$. Thus, $F\in \Delta_{j-1}$, i.e., $F\subset F_i$ for some $i\in \Np$ with $i< j$. Thus $F\subseteq F_i\cap F_j$, and since $\dim F = d-1$, we see that $F = F_i\cap F_j$. Now as noted in the previous paragraph, we can write $F_i \cap F_j = F_k \cap F_j$ for some $k \in J'$ with $k < j$. In particular, 
$F$ is a subspace of $F_k$ and so $F\in \Delta'$. This proves that $\ring{\Sigma}(F_j) \setminus \{F_j\} \subseteq \ring{\Sigma}(F_j) \cap \ring{\Delta}^{\prime}$. Consequently, $ \ring{\Sigma}(F_j) \cap \ring{\Delta}^{\prime}$ is the $q$-sphere  $\ring{\Sigma}(F_j) \setminus \{F_j\}$ of dimension $d-1$.
\end{proof}

\begin{rem}
Consider the shellable $q$-complex $\Delta_{C}$ of Example \ref{exa:6.4}. We have seen that the converse of 
Lemma \ref{lem:5.13} is not true for this. We have also seen that $\Delta_C$ has $14$ facets $F_1, \dots , F_{14}$, and we have determined the sets $\R_i(F_8)$ for $1\le i \le 7$. The remaining sets 
$\R_i(F_j)$ can also be easily computed. 
We are of course mainly interested in the unions $\mathscr{R}_j:= \cup_{i=1}^{j-1} \R_i(F_j)$ for $2\le j \le 14$, and it turns out that 
\begin{eqnarray*}
\mathscr{R}_2 &=& \{ \mathbf{e}_1 +\mathbf{e}_2,\; \mathbf{e}_1+\mathbf{e}_2+\mathbf{e}_3,\; \mathbf{e}_1+\mathbf{e}_2+\mathbf{e}_4,\;\mathbf{e}_1+\mathbf{e}_2+\mathbf{e}_3+\mathbf{e}_4 \},\\
\mathscr{R}_3 &=& \{\mathbf{e}_1,\; \mathbf{e}_1+\mathbf{e}_2,\; \mathbf{e}_1+\mathbf{e}_4, \;\mathbf{e}_1+\mathbf{e}_2+\mathbf{e}_4,\; \mathbf{e}_2, \;\mathbf{e}_2+\mathbf{e}_4 \},\\
\mathscr{R}_4 &=& \{\mathbf{e}_1+\mathbf{e}_3, \mathbf{e}_1+\mathbf{e}_2+\mathbf{e}_3, \mathbf{e}_1+\mathbf{e}_3+\mathbf{e}_4, \mathbf{e}_1+\mathbf{e}_2+\mathbf{e}_3+\mathbf{e}_4, \mathbf{e}_2, \mathbf{e}_2+\mathbf{e}_4\},\\
\mathscr{R}_5 &=& \{\mathbf{e}_1, \mathbf{e}_1+\mathbf{e}_2+\mathbf{e}_3, \mathbf{e}_1+\mathbf{e}_4, \mathbf{e}_1+\mathbf{e}_2+\mathbf{e}_3+\mathbf{e}_4, \mathbf{e}_2+\mathbf{e}_3, \mathbf{e}_2+\mathbf{e}_3+\mathbf{e}_4\},\\
\mathscr{R}_6 &=& \{\mathbf{e}_1+\mathbf{e}_3, \mathbf{e}_1+\mathbf{e}_2, \mathbf{e}_1+\mathbf{e}_3+\mathbf{e}_4, \mathbf{e}_1+\mathbf{e}_2+\mathbf{e}_4, \mathbf{e}_2+\mathbf{e}_3, \mathbf{e}_2+\mathbf{e}_3+\mathbf{e}_4\},\text{ and} \\
\mathscr{R}_j &=& F_j \backslash \{ \mathbf{0}  \} \text{ for } j=7, \ldots, 14.
\end{eqnarray*}
It can thus be seen that if $J$ and $J'$ are as in \eqref{eq:JDelta} with $\Delta= \Delta_C$, then 
 $$
 J = \{ 7,8, 9,10, 11, 12, 13, 14\} \quad \text{and} \quad J' = \{1,2,3,4, 5,6\}.
$$
Moreover, it is clear from the description in Example \ref{exa:6.4} of the facets $F_1, \dots , F_{14}$ of $\Delta_C$ that $\< \mathbf{e}_4 \> \subseteq F_j$ for  all $j\in J'$. Thus, by Lemma \ref{pro:::6}, $\Delta'_C:= \< F_j : j\in J'\>$ is acyclic. On the other hand, it can be seen that  for this $q$-complex of dimension $3$, 
$$
\ring{\Sigma}(F_7) \cap \ring{\Delta}^{\prime}_C =\ring{\Sigma}(F_7) \backslash I_7 =\ring{\Sigma}(F_7)\, \backslash \,\{\langle \mathbf{e}_1, \mathbf{e}_3 \rangle, F_7\},
$$ 
and this is not a punctured $q$-sphere of dimension $2$. Thus, we see that $\Delta_{C}$ does not satisfy one of  the hypotheses of Theorem \ref{thm:5.18}. The determination of singular homology of shellable $q$-complexes such as $\Delta_{C}$, which do not satisfy the hypothesis of Theorem~\ref{thm:5.18}, remains an open question. 
\end{rem}

\end{document}